\documentclass[10pt,letterpaper,reqno]{amsart}
\usepackage[latin9]{inputenc}
\usepackage[all]{xy}
\usepackage{bbm}
\usepackage{mathrsfs}
\usepackage{amsfonts}
\usepackage{bm}
\usepackage{amscd}
\usepackage{amsmath,amssymb,amsthm,stmaryrd}
\usepackage{enumerate}

\usepackage{hyperref}

\usepackage{longtable}

\usepackage{bussproofs}

\usepackage[numbers, sort&compress]{natbib}
\usepackage{mathbbol}

\swapnumbers

\newcommand{\nocontentsline}[3]{}
\newcommand{\tocless}[2]{\bgroup\let\addcontentsline=\nocontentsline#1*{#2}\egroup}

\newtheorem{teorema}{Theorem}[subsection]

\newtheorem*{teoA}{Theorem A}
\newtheorem*{teoB}{Theorem B}
\newtheorem*{teoC}{Theorem C}

\newtheorem{teorem}[teorema]{}
\newtheorem{corolario}[teorema]{Corollary}
\newtheorem{lema}[teorema]{Lemma}
\newtheorem{proposicion}[teorema]{Proposition}
\newtheorem{prop}{Proposition}[section]

\theoremstyle{definition}

\newtheorem{definicion}[teorema]{Definition}
\newtheorem{notacion}[teorema]{Notation}
\newtheorem{observacion}[teorema]{Remark}

\newcommand{\overbar}[1]{\mkern 1.5mu\overline{\mkern-1.5mu#1\mkern-1.5mu}\mkern 1.5mu}
\newcommand{\dom}{\mathrm{dom}}
\newcommand{\cod}{\mathrm{cod}}
\newcommand{\E}{\mathscr{E}}
\newcommand{\typea}{\mathbf{A}}
\newcommand{\typef}{\mathbf{f}}
\newcommand{\llocal}{\mathcal{L}}
\newcommand{\topos}{\mathscr{T}}

\newcommand{\teo}{\mathrm{Th}}

\newcommand{\eq}{\mathrm{eq}}
\newcommand{\uno}{\mathbf{1}}
\def\C{\mathscr C}
\newcommand{\Pot}{\mathbf{P}}
\newcommand{\Sub}{\mathrm{Sub}}
\newcommand{\subob}{\mathit{\Omega}}
\newcommand{\FE}{\mathbb{E}}

\begin{document}
\title[On logical parameterizations and functional representability in LST]{On logical parameterizations and functional representability in local set theories}
\author[Ruiz]{Enrique Ruiz Hern\'andez${}^\dagger$}
\address[$\dagger$]{Centro de Investigaci\'on en Teor\'ia de Categor\'ias y sus Aplicaciones, A.C. M\'e\-xi\-co.}
\curraddr{}
\email{e.ruiz-hernandez@cinvcat.org.mx}

\author[Sol\'orzano]{Pedro Sol\'orzano${}^\star$}
\address[$\star$]{Instituto de Matem\'aticas, Universidad Nacional Aut\'onoma de M\'exico, Oaxaca de Ju\'a\-rez, M\'e\-xi\-co.}
\curraddr{}
\email{pedro.solorzano@matem.unam.mx}
\thanks{($\star$) {CONACYT-UNAM} Research fellow.}
\date{\today}

\begin{abstract} 
There is a well-known inclusion $\iota_\E$ of a topos $\E$ in the linguistic topos $\topos(\Sigma)$ of its internal language $\Sigma$ that proves both toposes to be equivalent. There is also a canonical translation $\eta_S$ for any local set theory $S$ into the local set theory $\Sigma$ of its linguistic topos. Starting from a local set theory, this yields two a priori distinct inclusions from $\topos(S)$ to $\topos(\Sigma)$. 

Herein, these two functors are proved to be isomorphic. Furthermore, the concept of logical parameterization is investigated and then applied to see that $\iota_{\topos(S)}$ parameterizes $\topos(\eta_S)$ in such a way that syntactic $S$-functions are represented by themselves in $\Sigma$. 
\end{abstract}

\maketitle

\tocless\section{Motivation}
A local language, as considered by Zangwill (see \cite{MR972257}), is a typed language $\llocal$ with primitives $\in$, $\{\,:\,\}$ and $=$; ground, product and power type symbols; distinguished type symbols $\bm 1, \bm \Omega$; and function symbols; together with a natural intuitionistic deduction system. A local set theory (LST) is a collection $S$ of sequents $\Gamma:\alpha$, closed under derivability. Thus, the resulting logic of an LST might still not be classical. A local set theory $S$ produces a category $\topos(S)$ with its syntactical sets and functions as in classical Set Theory. Said categories are also toposes. Conversely, starting with a topos, a local language can be produced---e.g. Mitchell-B\'enabou, Zangwill, Lambek-Scott (See \cite{MR0470019,MR972257,MR856915})---whose linguistic topos is equivalent to the starting topos.  

For any topos $\E$, denote by $\teo(\E)$ the LST associated to it. In particular, its type and function symbols are the objects and arrows of $\E$. There is a {\em canonical inclusion} $\iota_\E:\E\rightarrow\topos(\mathrm{Th}(\E))$, given by $\iota_\E X:=\{u_{\bm X}:\top\}=:U_{\bm X}$ and $\iota_\E f$ is the syntactic function $(x\mapsto\typef(x)):U_{\bm X}\bm{\rightarrow}U_{\bm Y}$.  The Equivalence Theorem (also cf. \cite[Proposition II.13.3]{MR856915}) shows that for any topos $\E$, $\iota_\E$ is an equivalence of categories. 

On the other hand, for any LST $S$ there is a natural translation $\eta_S:S\to\teo(\topos(S))$ which in particular sends every type symbol $\typea$ of $S$ to the type symbol corresponding to the universal object $U_{\typea}$ and every function symbol $\mathbf f$ to the function symbol corresponding to $(x\mapsto \mathbf f(x))$.  Since any translation induces a logical functor between the corresponding linguistic toposes, one has two a priori distinct functors:
\begin{equation}\label{eq:twoamigos}
\iota_{\topos(S)}, \topos(\eta_S):\topos(S)\rightarrow\topos(\mathrm{Th}(\topos(S))).
\end{equation}
From the syntactic viewpoint, the more intuitive of the two is $\topos(\eta_S)$, yet it is $\iota_{\topos(S)}$ the one used to prove the equivalence of categories. To the best of our knowledge, nowhere in the literature is any relationship between both of them asserted.

In an arbitrary LST, syntactic functions are not in general of the form $(x\mapsto \tau)$, where $\tau$ a valid term in the language with $x$ its free variable.  Bell resolves this for an individual syntactic function $f:X\to Y$ (and thus for finitely many) by manually changing both the language and the LST by adding a function symbol $\typef $ and the corresponding axioms to guarantee that the symbolic function $(x\mapsto\typef(x))$ coincides with the given syntactic function when restricted to $X$, i.e. by adding
\begin{equation}\label{eq:predic}
x\in X:\langle x,\typef(x)\rangle\in \lvert f\rvert.
\end{equation}
This procedure is a basic feature of everyday mathematical practice, yet it requires a change of topos every time a new function symbol is added.  
To a certain degree, the Equivalence Theorem provides a global alternative to the iterative procedure given by the addition of sequents of the form \eqref{eq:predic}. Within its proof, an inverse to $\iota$,
\begin{equation}\rho_\E:\topos(\mathrm{Th}(\E))\rightarrow\E,
\end{equation}
is constructed as follows: For $\bm{\{}x\bm{:}\alpha\bm{\}}$ of type $\mathbf{PA}$ in $\topos(\mathrm{Th}(\E))$, $\rho_\E \bm{\{}x\bm{:}\alpha\bm{\}}$ is the kernel of the interpretation $\llbracket\alpha\rrbracket_x$ of $\alpha$ in $\E$. For $\bm{\{}x\bm{:}\alpha\bm{\}}\overset{\varphi}{\bm{\rightarrow}}\bm{\{}y\bm{:}\beta\bm{\}}$ in $\topos(\mathrm{Th}(\E))$, $\rho_\E \varphi:\rho_\E \bm{\{}x\bm{:}\alpha\bm{\}}\bm{\rightarrow}\rho_\E \bm{\{}y\bm{:}\beta\bm{\}}$ is the unique $\psi$ such that 
\begin{equation}\label{eq:definingGf}
\vdash_{\mathrm{Th}(\E)}\langle\bm{i_{\{x:\alpha\}}}(u),\bm{i_{\{y:\beta\}}}(\bm{\psi}(u))\rangle\bm{\in}\lvert \varphi\rvert.
\end{equation}
This equation shows precisely the extent by which achieving strict representability (as the one obtained from \eqref{eq:predic}) fails: The representation of $\varphi$ as an explicit function $(x\mapsto \tau)$ is only fulfilled up to isomorphism. Also, the asserted existence of $\psi$, albeit constructive, is never traced. 

\

\tocless\section{Main Results}
The purpose of this communication is to investigate the categorical, logical, and syntactical relationship between the two functors in \eqref{eq:twoamigos}.

We start by giving an explicit description of $\psi$ in \eqref{eq:definingGf} in the case where $\varphi$ is of the form $\topos(\eta_S)f$, and as a consequence establishing the categorical isomorphism between the two functors in \eqref{eq:twoamigos}. Namely, 
\begin{teoA}\hypertarget{T:TeorA} 
Let $S$ be a local set theory in a local language $\llocal$. Let $f:X\rightarrow Y$ be an $S$-function with $X$ of type $\mathbf{PA}$ and $Y$ of type $\mathbf{PB}$. Let $i_X$ y $i_Y$ be the functions $(x\mapsto x):X\rightarrow U_{\typea}$ and  $(y\mapsto y):Y\rightarrow U_{\bm B}$, respectively.
Then, 
\begin{equation}\label{E:symboloffasfunction}
\vdash_{\teo(\topos(S))}\langle\bm{i_X}(u),\bm{i_Y}(\typef(u))\rangle\in\lvert \topos(\eta_S)f\rvert.
\end{equation}
That is, $f$ is the only $S$-function for which the following diagram commutes in $\topos(\teo(\topos(S)))$:
\begin{equation}\label{E:diagramanatural}
\xymatrix{
U_{\bm{X}}\ar[r]^{r_X}\ar[d]_{(x\mapsto\typef(x))} & \topos(\eta_S)X\ar[d]^{\topos(\eta_S)f}\\
U_{\bm{Y}}\ar[r]_{r_Y} & \topos(\eta_S)Y,
}
\end{equation}
where  $r_X$ is the corestriction of $(u\mapsto\bm{i_X}(u))$ to $\topos(\eta_S)X$.
In other words, the inclusion functor $\iota_{\topos(S)}$ and the logical functor induced by the canonical translation $\eta_S$ are isomorphic:
\[\topos(\eta_S)\cong \iota_{\topos(S)}.\]
\end{teoA} 
Once this was established, it was clear that a better understanding of how the change of variables was occurring from $\topos(\eta_S)$ to $\iota_{\topos(S)}$ was in order. The following two observations are needed to enunciate the next result.  Firstly, the eliminability of descriptions in $\teo\topos(S)$ gives that any function $f$ with universal domain $U_{\typea}$ in $\topos(\mathrm{Th}(\topos(S)))$ can indeed be represented by 
\[(x\mapsto \tau_f)\]
for some appropriate term $\tau_f$ (see \ref{L:corestrictiontoX}).  Secondly, for any $S$-function $f:X\rightarrow Y$ and for any formula $\vartheta$ in $\llocal$ let
\[\vartheta^\natural_f\equiv\exists y(\langle x,y\rangle\in\lvert f\rvert\wedge\vartheta).
\]
\begin{teoB} \hypertarget{T:TeorB}
Let $\Sigma$ be a well-termed local set theory in a local language $\llocal$. Let $\typea$  and $\mathbf{B}$ be types in $\llocal$. Let $f:U_{\mathbf{B}}\rightarrow X$ be a surjective $\Sigma$-function with $X$ determined by the formula $\alpha$ with unique variable of type $\typea$, and $\tau_f$ a term that represents $f$.  Let $\gamma$ and $\Gamma$ be a formula and a finite set of formulae in $\llocal$. Then, 
\[
\Gamma,x\in X\vdash_{\Sigma}\gamma\quad\text{ if and only if }\quad\Gamma(x/\tau_f)\vdash_{\Sigma}\gamma(x/\tau_f).
\]

Additionally, if $f$ is also injective, then, for every formula $\vartheta$, 
\[\vdash_{\Sigma}\vartheta=\vartheta^\natural_{f^{-1}}(x/\tau_f).\]
In this case, 
\[
\Gamma\vdash_{\Sigma}\gamma\quad\text{if and only if}\quad\Gamma^\natural_{f^{-1}},x\in X\vdash_{\Sigma}\gamma^\natural_{f^{-1}}, 
\]
and furthermore, 
\begin{enumerate} 
\item $\vdash_{\Sigma}(\vartheta\square\gamma)^\natural_{f^{-1}}=\vartheta^\natural_{f^{-1}}\square\gamma^\natural_{f^{-1}}$, where $\square$ is any binary logical connective; 
\item $\vdash_{\Sigma}(\neg\vartheta)^\natural_{f^{-1}}=\neg\vartheta^\natural_{f^{-1}}$; and
\item if $u$ is free in $\vartheta$ but $x$ is not,  
\[\vdash_{\Sigma}\forall u\vartheta=\forall x\in X.\vartheta^\natural_{f^{-1}}\text{ and }\vdash_{\Sigma}\exists u\vartheta=\exists x\in X.\vartheta^\natural_{f^{-1}}.
\] 
\end{enumerate}
\end{teoB}
When applied to $\Sigma=\teo(\topos(S))$, Theorem B provides an explicit natural parameterization of any $\topos\eta_S(X)$ by the universal $\iota_S(X)$ preserving all logical operations.

To see that these parameterizations are fully compatible with doing mathematics internally (or to further justify the use of $\iota_S$ instead of $\topos\eta_S$), it remains to see that they behave nicely with respect to functions.  To this effect, observe that since any $S$-function $f:X\to Y$ is determined by a formula $\gamma(x,y)$ in $S$, the formula $\topos(\eta_S)\gamma(x,y)$ determines the function $\topos(\eta_S)f$ in $\teo(\topos(S))$.  By considering the term
\[\topos(\eta_S)\gamma(x/\bm{i_X}(u),y/\bm{i_Y}(v)),\]
one would expect to obtain another function $f^*:U_{\bf X}\to U_{\bf Y}$.  Applying Theorem B gives the following result. 
\begin{teoC}\hypertarget{T:TeorC} 
Let $S$ be a local set theory in a local language $\llocal$. Let $f:X\rightarrow Y$ be an $S$-function, and let $\gamma$ be a formula that determines $\lvert f\rvert$. Then 
\begin{equation}\label{E:fstar1} \vdash_{Th(\topos(S))}(\topos(\eta_S)\gamma(x/\bm{i_X}(u),y/\bm{i_Y}(v)))^\natural_{r_X^{-1}\times r_Y^{-1}}=\topos(\eta_S)\gamma,
\end{equation}
and the formula $\topos(\eta_S)\gamma(x/\bm{i_X}(u),y/\bm{i_Y}(v))$ indeed determines a function $f^*:U_{\bf X}\to U_{\bf Y}$ which furthermore coincides with $(u\mapsto \typef(u))$. In other words,
\begin{equation}
\vdash_{Th(\topos(S))} \langle u,\typef(u) \rangle\in \lvert f^*\rvert.
\end{equation}
\end{teoC}
The structure of this report is as follows. In Section \ref{S:Prelim}, we review all the prerequisites of Zangwill's local set theories. In Section \ref{S:Param}, the notion of logical parameterization is investigated in order to prove \hyperlink{T:TeorB}{Theorem B}. Lastly, in Section \ref{s:representabilidad}, \hyperlink{T:TeorA}{Theorem A} and \hyperlink{T:TeorC}{Theorem C} are proved. The latter is obtained by applying the techniques developed in Section \ref{S:Param}. A brief Appendix is included to further describe the tight syntactic relationship between $S$ and $\teo(\topos(S))$. 

\subsection*{Aknowledgements.} The second named author is supported by the CONAHCYT ``Investigadoras e Investigadores por M\'exico'' Program Project No. 61. The authors wish to thank an anonymous referee for greatly helping improve the presentation of the main results.

\section{Preliminaries}\label{S:Prelim}
A working knowledge of toposes and local set theories is assumed in this report. A detailed introduction to both subjects is given in \cite{MR972257}. The expert may proceed directly to Section \ref{S:Param}, and return to this section only when necessary.

The purpose of this section is to recall those basic features of these fields that will be required in what follows, as well as to fix the notation to be used henceforth. Except where noted, all the results are already in \cite{MR972257}. Therein, some of these results are given for very weak assumptions. Here they are recalled with the assumptions required. 

\subsection{Toposes} A topos is a category $\C$ with finite products, a subobject classifier and power objects. A subobject of an object $A$ is an equivalence class $[m]$ of monic maps $m$ with codomain $A$ under an appropriate relation. A subobject classifier is an object $\subob$  together with a map $\top:1\rightarrow\subob$ such that for any monic $m$ there exists a unique characteristic $\chi(m)$ such that the following is a pullback diagram.
\[
\xymatrix{
\dom(m)\ar[r]\ar[d]_m & 1\ar[d]^\top \\
\cod(m)\ar[r]_{\chi(m)} & \mathit{\Omega}
}
\]
And every $A\to\subob$ is of this form.  

Given an object $A$ of $\C$, a power object is a pair $(PA,e_A)$ of the object $PA$ and an arrow $e_A:A\times PA\rightarrow\subob$ such that for every $f:A\times B\rightarrow\subob$ there is a unique arrow $\hat{f}:B\rightarrow PA$ such that the following diagram commutes.
\begin{equation}\label{E:transposepower}\xymatrix{
A\times B\ar[dr]^f\ar[d]_{1_A\times\hat{f}} & \\
A\times PA\ar[r]_(.6){e_A} & \subob.
}
\end{equation}

In a topos $\C$, given an object $A$, there is a bijection between the class of $\Sub(A)$ of subobjects of $A$ and the class $\C(A,\subob)$ of arrows $A\rightarrow\subob$ in $\C$. 
Given a monic $m$ representing a subobject of $A$, its corresponding characteristic is $\chi(m):A\rightarrow\subob$; and given an arrow $u:A\rightarrow\subob$ in $\C$, its corresponding subobject is represented by $\bar{u}$.  It follows that  $u=\chi(\bar{u})$; $\left[\overbar{\chi(m)}\right]=[m]$; and that $[m]=[n]$ if and only if $\chi(m)=\chi(n)$.

From this, one can transfer the partial order $\subseteq$ unto $\C(A,\Omega)$, so that given $u,v\in \C(A,\subob)$, $u\leq v$ if and only if $\bar{u}\subseteq\bar{v}$. In particular,  this is equivalent to the existence of an arrow $f:\dom(\bar{u})\rightarrow\dom(\bar{v})$ such that
\begin{equation}\label{E:transferedorder}
\bar{u}=\bar{v}\circ f.
\end{equation}
Recall that the maximal characteristic arrow of an object $A$, denoted by $T_A$, is $T_A:=\chi(1_A):A\rightarrow\subob$. Wherever $A$ is understood, one writes simply $T$. Notice that $\overbar{T_A}=1_A$, and that $T_A$ is maximal since  $u\leq T_A$ for all $u:A\rightarrow\subob$.  Furthermore, recall that $T=v\circ\bar{u}$ if and only if $u\leq v$ and that 
\begin{equation}\label{E:relofmaximal}
T_C=v\circ f\text{ if and only if a }g:C\rightarrow\dom(\bar{v})\text{ exists with }f=\bar{v}\circ g.
\end{equation}




\subsection{Local languages}
The central definition of this communication is the definition of a local language.  A {\em local language} $\llocal$ is determined by the specification of the following classes of symbols: a true-value symbol $\mathbf{\Omega}$ and a unity type symbol $\uno$; a possibly empty collection of ground type symbols $\mathbf{A,B,C,}\ldots$; a possibly empty collection of function symbols $\mathbf{f,g,h,}\ldots$. The type symbols of $\llocal$ are now defined recursively as follows: $\uno,\mathbf{\Omega}$ are type symbols; any ground type symbol is a type symbol; if $\typea_1,\ldots,\typea_n$ are type symbols, so is the product $\typea_1\times\cdots\times\typea_n$ (with the proviso that, if $n=1$, then $\typea_1\times\cdots\times\typea_n$ is $\typea_1$, and if $n=0$, then $\typea_1\times\cdots\times\typea_n$ is $\uno$); and if $\typea$ is a type symbol, so is the power $\Pot\typea$.

For each type symbol $\typea$, $\llocal$ contains a denumerable set of symbols $x_\typea,y_\typea$, $z_\typea,\ldots$ called \textit{variables} of type $\typea$. In addition, $\llocal$ contains the symbol $\ast$. Each function symbol of $\llocal$ has a signature assigned of the form $\typea\rightarrow\mathbf{B}$, where $\typea,\mathbf{B}$ are type symbols. 

The terms with their corresponding types are recursively defined as follows: $\ast$ is a term of type $\uno$; for each type symbol $\typea$, variables $x_\typea,y_\typea,z_\typea,\ldots$ are terms of
type $\typea$; if $\mathbf{f}$ is a function symbol of signature $\typea\rightarrow\mathbf{B}$, and $\tau$ is a term of type $\typea$, $\mathbf{f}(\tau)$ is a term of type $\mathbf{B}$; if $\tau_1,\ldots,\tau_n$ are terms of types $\typea_1,\ldots,\typea_n$, $\langle\tau_1,\ldots,\tau_n\rangle$ is a term of type $\typea_1\times\cdots\times\typea_n$, with the proviso that if $n=1$, then $\langle\tau_1,\ldots,\tau_n\rangle$ is $\tau_1$, while if $n=0$, then $\langle\tau_1,\ldots,\tau_n\rangle$ is $\ast$; if $\tau$ is a term of type $\typea_1\times\cdots\times\typea_n$ and $1\leq i\leq n$, $(\tau)_i$ is a term of type $\typea_i$, with the condition that if $n=1$, then $(\tau)_1$ is $\tau$; if $\alpha$ is a term of type $\mathbf{\Omega}$, and $x_\typea$  is a variable of type $\typea$, $\{x_\typea:\alpha\}$ is a term of type $\Pot\typea$; if $\sigma,\tau$ are terms of the same type, $\sigma=\tau$ is a term of type $\mathbf{\Omega}$; and if $\sigma,\tau$ are terms of types $\typea$, $\Pot\typea$ respectively, $\sigma\in\tau$ is a term of type $\mathbf{\Omega}$.

A term of type $\mathbf{\Omega}$ is called a {\em formula}.  Arbitrary terms will be denoted by Greek letters $\sigma,\tau,\ldots$. Symbols $\omega,\omega',\omega'',\ldots$ will be reserved for variables of type $\mathbf{\Omega}$ and $\alpha,\beta,\gamma,\ldots$ for formulae.  An occurrence of a variable $x$ in a term $\tau$ is {\em bound} if it appears within a context of the form $\{x:\alpha\}$; otherwise the occurrence is {\em free}. A term with no free variables is said to be {\em closed}; a closed formula is called a {\em sentence}. For any terms $\tau,\sigma$ and any variable $x$ of the same type as $\sigma$ one writes $\tau(x/\sigma)$ (or sometimes just $\tau(\sigma)$) for the term obtained from $\tau$ by substituting $\sigma$ for each free occurrence of $x$. 

Similarly, for any variables $x_1,\ldots,x_n$ and terms $\sigma_1,\ldots,\sigma_n$ of the appropriate types, write $\tau(x_1/\sigma_1,\ldots,x_n/\sigma_n)$ or briefly $\tau(\bm{x}/\bm{\sigma})$ for the result of substituting $\sigma_i$ for $x_i$ in $\tau$ for $1\leq i\leq n$. 

Logical operations are defined in $\llocal$ as follows: Write

\begin{center}
\begin{tabular}{l l c l}
(L1) & $\alpha\Leftrightarrow\beta$ & for & $\alpha=\beta$ \\
(L2) & \textit{true} & for & $\ast=\ast$ \\
(L3) & $\alpha\wedge\beta$ & for & $\langle\alpha,\beta\rangle=\langle true,true\rangle$ \\
(L4) & $\alpha\Rightarrow\beta$ & for & $(\alpha\wedge\beta)\Leftrightarrow\alpha$ \\
(L5) & $\forall x\alpha$ & for & $\{x:\alpha\}=\{x:true\}$ \\
(L6) & \textit{false} & for & $\forall\omega.\omega$ \\
(L7) & $\neg\alpha$ & for & $\alpha\Rightarrow false$ \\
(L8) & $\alpha\vee\beta$ & for & $\forall\omega((\alpha\Rightarrow\omega\wedge\beta\Rightarrow\omega)\Rightarrow\omega)$ \\
(L9) & $\exists x\alpha$ & for & $\forall\omega(\forall x(\alpha\Rightarrow\omega)\Rightarrow\omega)$
\end{tabular}
\end{center}

In (L8) and (L9) $\omega$ is assumed to be of type $\mathbf{\Omega}$ not occurring in $\alpha$ or $\beta$. The unique description  $\exists x(\alpha\wedge\forall y(\alpha(x/y)\Rightarrow x=y))$ is abbreviated as $\exists!x\alpha$, where $y$ is different from $x$ and not free in $\alpha$.

The basic axioms, inference rules and thus all deductions in a local set theory are given in terms of {\em sequents} $\Gamma:\alpha$, where $\alpha$ is a formula and $\Gamma$ is a (possibly empty) finite set of formulae.  The {\em basic axioms} are the sequents:  $\alpha:\alpha$ (Tautology); $:x_\uno=\ast$ (Unity); $x=y,\alpha(z/x):\alpha(z/y)$ with $x$, $y$ free for $z$ in $\alpha$ (Equality);  $:(\langle x_1,\ldots,x_n\rangle)_i=x_i$, $:x=\langle(x)_1,\ldots,(x)_n\rangle$ (Products); and $:x\in\{x:\alpha\}\Leftrightarrow\alpha$ (Comprehension).  

The inference rules are the following:

\begin{longtable}{l p{9cm}l}
\textit{Thinning} & \AxiomC{$\Gamma:\alpha$}\UnaryInfC{$\beta,\Gamma:\alpha$}\DisplayProof \\
\multicolumn{2}{c}{} \\
\textit{Cut} & \AxiomC{$\Gamma:\alpha$}\AxiomC{$\alpha,\Gamma:\beta$}\BinaryInfC{$\Gamma:\beta$}\DisplayProof (free variables in $\alpha$ free in $\Gamma$ or $\beta$) \\
\multicolumn{2}{c}{} \\
\textit{Substitution} & \AxiomC{$\Gamma:\alpha$}\UnaryInfC{$\Gamma(x/\tau):\alpha(x/\tau)$}\DisplayProof ($\tau$ free for $x$ in $\Gamma$ and $\alpha$) \\
\multicolumn{2}{c}{} \\
\textit{Extensionality}\hspace{-50pt}& \AxiomC{$\Gamma:x\in\sigma\Leftrightarrow x\in\tau$}\UnaryInfC{$\Gamma:\sigma=\tau$}\DisplayProof ($x$ not free in $\Gamma,\sigma,\tau$) \\
\multicolumn{2}{c}{} \\
\textit{Equivalence} & \AxiomC{$\alpha,\Gamma:\beta$}\AxiomC{$\beta,\Gamma:\alpha$}\BinaryInfC{$\Gamma:\alpha\Leftrightarrow\beta$}\DisplayProof
\end{longtable}

\

For any collection $S$ of sequents, a {\em proof from $S$} is a finite tree with its vertex at the bottom whose nodes are correlated downwards by direct consequence by one of the rules of inference and topmost nodes are correlated with either a basic axiom or a member of $S$. The sequent at the vertex is the {\em conclusion} of the proof.   A sequent $\Gamma:\alpha$ is {\em derivable from $S$}, and write $\Gamma\vdash_S\alpha$ provided there is a proof from $S$ of which the sequent $\Gamma:\alpha$ is the conclusion. Following \cite{MR972257}, successive applications of the cut rule will be written in the form $\Gamma\vdash_S\alpha_1\vdash_S\alpha_2\vdash_S\cdots\vdash_S\alpha_n$ yielding $\Gamma\vdash_S\alpha_n$.

In \cite{MR972257} it is proved that all the logical operations (L1)-(L9) satisfy the laws of intuitionistic logic. For convenience, the following is a list of some of those that are being used in the sequel. 
\begin{teorem}[\cite{MR972257} 3.2.3]\label{t:implicationinhypotheses}
(i)\AxiomC{$\Gamma:\alpha$}\AxiomC{$\beta,\Gamma:\gamma$}\BinaryInfC{$\alpha\Rightarrow\beta,\Gamma:\gamma$}\DisplayProof,\qquad (ii)\AxiomC{$\Gamma:\alpha\Rightarrow\beta$}\UnaryInfC{$\Gamma,\alpha:\beta$}\DisplayProof.
\end{teorem}
\begin{teorem}[\cite{MR972257} 3.4.2]\label{t:universalizationontheright}
\AxiomC{$\Gamma:\alpha$}\UnaryInfC{$\Gamma:\forall x\alpha$}\DisplayProof (provided either (i) $x$ is not free in $\Gamma$ or (ii) $x$ not free in $\alpha$).
\end{teorem}
\begin{teorem}[\cite{MR972257} 3.4.3]\label{t:samesetsamefourmulas}
\AxiomC{$\Gamma:\{x:\alpha\}=\{x:\beta\}$}\UnaryInfC{$\Gamma:\alpha\Leftrightarrow\beta$}\DisplayProof (with $x$ free in $\alpha$ or $\beta$).
\end{teorem}
\begin{teorem}[\cite{MR972257} 3.4.4]\label{t:universalelimination}
$\forall x\alpha\vdash\alpha$ (with $x$ free in $\alpha$).
\end{teorem}
\begin{teorem}[\cite{MR972257} 3.7.1]\label{t:existentialintroduction}
$\alpha\vdash\exists x\alpha$ (with $x$ free in $\alpha$).
\end{teorem}
\begin{teorem}[\cite{MR972257} 3.7.2]\label{t:existentializationontheleft}
\AxiomC{$\alpha,\Gamma:\beta$}\UnaryInfC{$\exists x\alpha,\Gamma:\beta$}\DisplayProof (provided either that (i) $x$ is not free in $\Gamma$ or $\beta$; or (ii) $x$ not free in $\alpha$).
\end{teorem}
\begin{teorem}[\cite{MR972257} 3.7.4]\label{t:existsubstiright}
\AxiomC{$\Gamma:\alpha(x/\tau)$}\UnaryInfC{$\Gamma:\exists x\alpha$}\DisplayProof (provided $\tau$ is free for $x$ in $\alpha$, $x$ free in $\alpha$ and every variable of $\tau$ free in $\Gamma$ or $\exists x\alpha$).
\end{teorem}
\begin{teorem}[\cite{MR972257} 3.7.6]\label{t:slidingexistence}
$\vdash\exists x(\alpha\wedge\beta)\Leftrightarrow\alpha\wedge\exists x\beta$ (with $x$ is free in $\beta$ but not in $\alpha$).
\end{teorem}
\begin{teorem}[\cite{MR972257} 3.7.8]\label{t:unrestrictedcut}
\AxiomC{$\Gamma:\alpha$}
\AxiomC{$\alpha,\Gamma:\beta$}
\BinaryInfC{$\Gamma:\beta$}
\DisplayProof (provided there is a closed term of type $\typea$ whenever $\typea$ is a type of a free variable of $\alpha$ not free in $\Gamma$ or $\beta$).
\end{teorem}

\subsection{Local set theories}

A {\em local set theory} (LST) in $\llocal$ is a collection S of sequents which is closed under derivability: $\Gamma\vdash_S\alpha$ if and only if $(\Gamma:\alpha)$ is in $S$.  Clearly, any collection of sequents generates an LST.  A local set theory is said to be consistent if it is not the case that $\vdash_S false$.

Closed terms of power type are called $\llocal$-sets. Capital letters $A,B,\ldots,Z$ will be used to denote sets and the usual abbrevations will be used: 
$\forall x\in X.\alpha$ for $\forall x(x\in X\Rightarrow\alpha)$; $\exists x\in X.\alpha$  for  $\exists x(x\in X\wedge\alpha)$; $\exists!x\in X.\alpha$  for  $\exists!x(x\in X\wedge\alpha)$; and $\{x\in X:\alpha\}$ for  $\{x:x\in X\wedge\alpha\}$.

The expected set operations $\subseteq, \cap, \cup\ldots$ can be defined in the usual way for sets of the same type. For each type $\typea$ there is a universal set $\{x_\typea:true\}$ (resp. empty set $\{x_\typea:false\}$) denoted by $U_\typea$ or $A$ (resp. $\varnothing_\typea$ or $\varnothing$).  For arbitrary sets, products and disjoint unions can be defined in the usual way. Also, function spaces $\{u:u\subseteq Y\times X\wedge\forall y\in Y\exists!x\in X.\langle y,x\rangle\in u\}$ are denoted by $X^Y$. Lastly, the familiar  $\{\tau\}$ and $\{\tau:\alpha\}$ symbolize  $\{x:x=\tau\}$ and  $\{z:\exists x_1\cdots\exists x_n(z=\tau\wedge\alpha)\}$, respectively. 

An $S$-set is an equivalence class of $\llocal$-sets under the relation $\textstyle\sim_S$, given by $X\sim_S Y$ if and only if $\vdash_SX=Y$.  Representative $\llocal$-sets will sometimes be regarded as $S$-sets. 

An $S$-function $X\overset{f}{\rightarrow}Y$ is a triplet of $S$-sets $(\lvert f\rvert,X,Y)$ such that $\vdash_S\lvert f\rvert\in Y^X$, where $\lvert f\rvert$ is called the graph of $f$ as per usual. 
\begin{proposicion}\label{P:functionimpliesy=z}
Let $X,Y,\lvert f\rvert$ be $S$-sets with $\vdash_S\lvert f\rvert\subseteq X\times Y$. Then, $x\in X\vdash_S\exists!y\in Y(\langle x,y\rangle\in\lvert f\rvert)$ if and only if $x\in X\vdash_S\exists y\in Y(\langle x,y\rangle\in\lvert f\rvert)$ and $\langle x,y\rangle\in\lvert f\rvert,\langle x,y'\rangle\in\lvert f\rvert\vdash_Sy=y'$.
\end{proposicion}
\begin{proof}
Straightforward.
\end{proof}
For $S$-functions $\xymatrix{X\ar[r]^f & Y\ar[r]^g & Z}$ the composition is defined through its graph. Namely,
\begin{equation}\label{E:compfun}
\lvert g\circ f\rvert=\{\langle x,z\rangle:\exists y(\langle x,y\rangle\in f\wedge\langle y,z\rangle\in g)\}.
\end{equation}

A consequence of extensionality is that  for $S$-functions $X\overset{f}{\rightarrow}Y,X\overset{g}{\rightarrow}Y$, 
\begin{equation}\label{E:equalfunctions}
f\text{ is equal to }g\text{ if and only if }x\in X\vdash_S\langle x,y\rangle\in\lvert f\rvert\Leftrightarrow\langle x,y\rangle\in\lvert g\rvert.
\end{equation}

It can be verified that the collection of $S$-sets and $S$-functions forms a category $\mathscr{T}(S)$ that is also a topos, the {\em linguistic topos} of $S$. 

\subsection{Representability}\label{SS:termsinducefunctions} Let $X$ and $Y$ be $S$-sets and $\tau$ be a term for which 
\[
\langle x_1,\ldots,x_n\rangle\in X\vdash_S\tau\in Y,
\]
with all the free variables of $\tau$ contained within the $x_1,\ldots,x_n$.  It is immediate to see that $\{\langle\langle x_1,\ldots,x_n\rangle,\tau\rangle:\langle x_1,\ldots,x_n\rangle\in X\}$ is the graph of a function which will be denoted by $(\langle x_1,\ldots,x_n\rangle\mapsto\tau)$.

Not every $S$-function is representable in this way.   However, for an $S$-function $f:X\rightarrow\subob$, let\footnote{Bell uses the symbol $f^*$ but here it is reserved for the parameterizations of Subsection \ref{SS:param}.} 
\begin{equation}\label{E:omegadirectimage}
f^\natural(x):=\langle x,true\rangle\in\lvert f\rvert.
\end{equation}
Then,
\begin{align}
x\in X &\vdash\langle x,f^\natural(x)\rangle\in\lvert f\rvert\notag,
\end{align}
and therefore 
\begin{proposicion}[\cite{MR972257}, 3.16.4]\label{P:omegafunctions}
\

\begin{enumerate}[(i)]
 \item $\xymatrix{X\ar[rr]^(.48){x\mapsto f^\natural(x)} & & \mathit{\Omega}}=f$,
 \item $\vdash_S(\langle x_1,\ldots,x_n\rangle\mapsto\alpha)^\natural(\langle x_1,\ldots,x_n\rangle)\Leftrightarrow\alpha$.
\end{enumerate}
\end{proposicion}
 An LST is said to be {\em well-termed} if whenever $\vdash_S\exists!x\alpha$, there is a term $\tau$ such that its free variables are that of  $\alpha$ except $x$ and $\vdash_S\alpha(x/\tau)$; an LST for which every syntactic set is isomorphic to a universal is said to be {\em well-typed}.  For a given LST $S$, let $\mathbf{C}^\ast(S)$ be the subcategory of $S$-sets of the form $U_\typea$ and $S$-functions of the form $(x\mapsto\tau)$ with $\tau$ a term. It follows that $S$ is well-termed if and only if $\mathbf{C}^\ast(S)\hookrightarrow\topos(S)$ is full; and well-typed if and only it is dense. 

\subsection{Interpretations} An {\em interpretation} of a local language $\llocal$  is a pair $(\E,I)$, with $\E$ a topos and $I$ an assignment to each type symbol $\mathbf{\Xi}$ of an object  $\mathbf{\Xi}_I=\Xi$ in $\E$ in a way compatible with products and powers that further assigns to $\mathbf{\Omega}$ a subobject classifier $\mathbf{\Omega}_I=\subob$; to $\uno$ a terminal object $\uno_I=1$; and to a function symbol $\typef$ of signature $\typea\rightarrow\mathbf{B}$ an arrow in $\E$ of the form $\typef_I:\typea_I\to\mathbf{B}_I$.

Given any interpretation, terms are further interpreted as arrows in the following way. Let $\tau$ be a term of type $\mathbf{B}$ and let $\bm{x}=x_1,\ldots,x_n$  be different variables of types $\typea_1,\ldots,\typea_n$ including all free variables of $\tau$. Define 
\[\llbracket\tau\rrbracket_{I,x_1,\ldots,x_n}=\llbracket\tau\rrbracket_{I,\bm{x}}=\llbracket\tau\rrbracket_{\bm{x}}\]
as an arrow $A_1\times\cdots\times A_n\to B$ recursively by requiring that:  (i) $\llbracket\ast\rrbracket_{\bm{x}}$ be the unique $A_1\times\cdots\times A_n\rightarrow 1$; (ii) $\llbracket x_i\rrbracket_{\bm{x}}=\pi_i$; (iii) $\llbracket\typef(\tau)\rrbracket_{\bm{x}}=\typef_I\circ\llbracket\tau\rrbracket_{\bm{x}}$; (iv) $\llbracket\langle\tau_1,\ldots,\tau_n\rangle\rrbracket_{\bm{x}}=\langle\llbracket\tau_1\rrbracket_{\bm{x}},\ldots,\llbracket\tau_n\rrbracket_{\bm{x}}\rangle$; (v) $\llbracket(\tau)_i\rrbracket_{\bm{x}}=\pi_i\circ\llbracket\tau\rrbracket_{\bm{x}}$; (vi) $\llbracket\{y:\alpha\}\rrbracket_{\bm{x}}=(\llbracket\alpha(y/u)\rrbracket_{u,\bm{x}}\circ\mathrm{can})\hat{}$,
 where $u$ is of type $\mathbf{C}$, different from $x_1,\ldots,x_n$, free for $y$ in $\alpha$, $y$ is of type $\mathbf{C}$, and `$\mathrm{can}$' is the canonical isomorphism $C\times(A_1\times\cdots\times A_n)\cong C\times A_1\times\cdots\times A_n$ (cf. \eqref{E:transposepower}); (vii) $\llbracket\sigma=\tau\rrbracket_{\bm{x}}=\eq_C\circ\llbracket\langle\sigma,\tau\rangle\rrbracket_{\bm{x}}$, where $\eq_C=\chi(\delta_C)$ and $\delta_C$ is the diagonal in $C\times C$; and (viii) $\llbracket\sigma\in\tau\rrbracket_{\bm{x}}=e_C\circ\llbracket\langle\sigma,\tau\rangle\rrbracket_{\bm{x}}$ (cf. \eqref{E:transposepower}).

In particular, if $\tau$ is a term all of whose variables are among $z_1,\ldots,z_m$ and $\sigma_1,\ldots,\sigma_m$ are terms such that $\sigma_i$ is free for $z_i$ in $\tau$ for each $i=1,\ldots,m$; then, for each interpretation of $\llocal$ in a topos $\E$, 
\begin{equation}\label{E:subst}
\llbracket\tau(\bm{z}/\bm{\sigma})\rrbracket_{\bm{x}}=\llbracket\tau\rrbracket_{\bm{z}}\circ\llbracket\langle\sigma_1,\ldots,\sigma_m\rangle\rrbracket_{\bm{x}}
\end{equation}
where $\bm{x}$ includes all the free variables of $\sigma_1,\ldots,\sigma_m$.

The interpretation of a finite collection $\Gamma$ of formulae, $\llbracket\Gamma\rrbracket_{I,\bm{x}}$ is the interpretation of its conjunction. So that if $\beta$ is another formula and $\bm{x}=x_1,\ldots,x_n$ is the list of the free variables of $\Gamma\cup\{\beta\}$, then write
\begin{equation}\label{E:validsequent}
\Gamma\vDash_I\beta\quad\text{or}\quad\Gamma\vDash_\E\beta\quad\text{instead of}\quad\llbracket\Gamma\rrbracket_{I,\bm{x}}\leq\llbracket\beta\rrbracket_{I,\bm{x}}.
\end{equation} 
A sequent $\Gamma:\beta$ is {\em valid} under $I$ (or in $\E$) if $\Gamma\vDash_I\beta$. It is clear that $\vDash_I\beta$ if and only if $\llbracket\beta\rrbracket_{\bm{x}}=T$. From which if $\E$ is degenerate then $\vDash_I\alpha$ for every formula $\alpha$ in $\llocal$. Thus it can be easily verified that

\begin{teorem}[\cite{MR972257}, 3.23.2]\label{P:validityTrelation}
$\Gamma\vDash_I\alpha$ if and only if $\llbracket\alpha\rrbracket_{\bm{x}}\circ\overbar{\llbracket\Gamma\rrbracket_{\bm{x}}}=T$.
\end{teorem}

\begin{teorem}[\cite{MR972257}, 3.23.3]\label{P:equalityofsymbols}
$\Gamma\vDash_I\sigma=\tau$ if and only if $\llbracket\sigma\rrbracket_{\bm{x}}\circ\overbar{\llbracket\Gamma\rrbracket_{\bm{x}}}=\llbracket\tau\rrbracket_{\bm{x}}\circ\overbar{\llbracket\Gamma\rrbracket_{\bm{x}}}$.
\end{teorem}

Given a local set theory $S$, the canonical interpretation $(\topos(S), K_S)$ of its local language is the one that assigns  $\typea_{K_S} =U_\typea$ for each type symbol $\typea\notag$ and  $\xymatrix{\typef_{K_S} =U_\typea\ar[rr]^(.6){x\mapsto\typef(x)} && U_\mathbf{B}}$  for each function symbol $\typef$ with signature $\typea\rightarrow\mathbf{B}$. It follows that
\begin{equation}\label{E:soundcompletenessforTS}
\Gamma\vDash_{K_S}\alpha\quad\text{ if and only if }\quad\Gamma\vdash_S\alpha,
\end{equation}
and that 
\begin{equation} \label{E:interpretationcorrespondstofunctions}\llbracket\tau\rrbracket_{K_S,\bm{x}}=(\bm{x}\mapsto\tau).
\end{equation}
In particular, the following expected result holds. 
\begin{lema}\label{L:Xisalphatoo}
Let $S$ be a local set theory in a local language $\llocal$. Let $\alpha$ be a formula with a unique free variable $x$ of type $\typea$.  The following is a pullback diagram.
\begin{equation}\label{E:pullbackalpha}
\xymatrix{
\{x:\alpha\}\ar[r]\ar[d]_{(x\mapsto x)} & 1\ar[d]^\top \\
A\ar[r]^{\llbracket\alpha\rrbracket_x}_{(x\mapsto\alpha)} & \subob
}
\end{equation}
\end{lema}
\begin{proof} Straightforward.
\end{proof}
\begin{observacion}\label{O:ix} The map $(x\mapsto x):X\longrightarrow A$ will frequently be denoted by  $i_X$.
\end{observacion}

\subsection{Internal language and local set theory of a topos} Let $\E$ be a topos with specified terminal $1_\E$, subobject classifier $\subob_\E$, products and power objects. The internal language of $\E$ is the local language $\llocal(\E)$ whose type symbols are precisely the objects $X$ of $\E$, denoted by $\mathbf X$, and whose function symbols the arrows $f$ of $\E$,  denoted by $\mathbf f$, with signature $\mathbf{dom(f)\to cod(f)}$.  The obvious natural interpretation $(\E, H_\E)$ of $\llocal(\E)$ is given by $\mathbf{X}_{H_\E}=X$ and $\typef_{H_\E}=f$. Equality, membership and comprehension in $\llocal(\E)$ are written in bold; i.e. $\bm{=},\bm{\in},\bm{\{\ :\ \}}$.

The local set theory of a topos, $\mathrm{Th}(\E)$, is determined by the requirement that  $\Gamma:\alpha$ is in $\mathrm{Th}(\E)$ if $\Gamma\vDash_{H_\E}\alpha$. It follows that
\begin{equation}\label{E:TeosiiE}
\Gamma\vdash_{\mathrm{Th}(\E)}\alpha\quad\text{ if and only if }\quad\Gamma\vDash_{H_\E}\alpha.
\end{equation}
This interpretation is compatible with the intuition, e.g. for any $\E$-arrow $f$, 
\begin{equation}\label{E:monicsinTeoE}
f\text{ is monic if and only if }\typef(x)\bm{=}\typef(y)\vdash_{\mathrm{Th}(\E)}x\bm{=}y.
\end{equation}

As a consequence of \eqref{E:subst},  \ref{P:validityTrelation}, and \eqref{E:relofmaximal} explicit term representation can be described categorically as follows.  

\begin{teorem}[\cite{MR972257} 3.34]\label{L:substintoposE}
The following are equivalent:
\begin{enumerate}[(i)]
 \item $\vdash_{\mathrm{Th}(\E)}\alpha(y_1/\tau_1,\ldots,y_n/\tau_n)$,
 \item $\overbar{\llbracket\alpha\rrbracket_{H_\E,\bm{y}}}\circ f=\llbracket\langle\tau_1,\ldots,\tau_n\rangle\rrbracket_{H_\E, \bm{x}}$ for some arrow $f$, where $y_1,\ldots,y_n$ are the free variables of $\alpha$ and $\tau_i$ is free for $y_i$ in $\alpha$ for each $i$.
\end{enumerate}
\end{teorem}
\begin{teorem}[\cite{MR972257} 3.35]\label{L:substintoposEbis} If $m$ is monic in $\E$, then
\[
\vdash_{\teo(\E)}\bm{\chi(m)}(x)\Leftrightarrow\exists y.x\bm{=m}(y).
\]

\end{teorem}

$\mathrm{Th}(\E)$ is both well-termed and well-typed as the following two results show.
\begin{teorema}[\cite{MR972257} 3.36]\label{T:eliminabilityinTeo}
If $\vdash_{\mathrm{Th}(\E)}\exists!y\alpha$, then there exists a unique $\E$-arrow $f$ such that
\[\vdash_{\mathrm{Th}(\E)}\alpha(y/\typef(\langle x_1,\ldots,x_n\rangle)),\]
where $x_1,\ldots,x_n,y$ are the free variables of $\alpha$.
\end{teorema}
Finally, passing from $\E$ to the linguistic topos of $\mathrm{Th}(\E)$ gives no additional information. Indeed, the following equivalence theorem holds. 
\begin{teorema}[\cite{MR972257} 3.37]\label{T:equivalencetheorem}
Every topos $\E$ is equivalent to the linguistic topos $\topos(\mathrm{Th}(\E))$  of its local set theory. 
\end{teorema}

\subsection{Translations}
A {\em translation} $\kappa:\llocal\to\Lambda$ between local languages $\llocal$ and $\Lambda$ is an assignment of a type (resp. function) symbol in $\Lambda$ to each type (resp. function) symbol in $\llocal$ in a way compatible with products, powers, truth-value and unity types. Terms are translated in the expected way as to preserve the basic axioms.  Between LSTs, $S$ and $\Sigma$ in $\llocal$ and $\Lambda$ respectively, $\kappa$ is said to be a translation if for every sequent $\Gamma:\alpha$ of $\llocal$, $\Gamma\vdash_S\alpha$ implies that $\kappa\Gamma\vdash_{\Sigma}\kappa\alpha$. If under the same conditions $\Gamma\vdash_S\alpha$ if and only if $\kappa\Gamma\vdash_{\Sigma}\kappa\alpha$, then $\kappa$ is called \textit{conservative}.

Any translation $\kappa:\llocal\to\Lambda$ induces a logical functor $\topos(\kappa)$  (one that preserves the specified terminal object, products, subobject classifier, power objects and evaluation arrows) between the associated linguistic toposes. 

For any local set theory $S$ with underlying local language $\llocal$, let the {\em canonical translation} $\eta_S$ be the assignment to each type symbol  $\mathbf{\Xi}$ of the type symbol associated to $U_\mathbf{\Xi}$  and to $\mathbf f$ of the function symbol associated to $(x\mapsto \mathbf f(x))$. It follows that it is a conservative translation and that 
\begin{equation}\label{E:canonical-natural}
\llbracket\tau\rrbracket_{K_S,\bm{x}}=\llbracket\eta_S(\tau)\rrbracket_{H_{\topos(S)},\eta_S\bm{x}}.
\end{equation}

Whenever there is no risk of confusion one can represent the type symbol associated to a type symbol $\typea$ with the same letter; in essence writing $\typea$ for $\eta_S(\typea)$. In particular, given an $S$-set $X:=\{x_\typea:\alpha\}$, $\eta_S(X)$ is described by $\{\eta_S(x_\typea):\eta_S(\alpha)\}$, and since $\eta_S$ is conservative, there is no loss in representing $\eta_S(X)=X$.   The same will be assumed for a syntactical function $f=(\lvert f\rvert,X,Y)$. However, $(x\mapsto\typef(x))$ need not be an $S$-function, since the function symbol $\typef$ need not be in $\llocal$, only in $\llocal(\topos(S))$. That is, $(x\mapsto\typef(x)):\bm{X\rightarrow Y}$ is an arrow in $\topos(\teo(\topos(S)))$, not in $\topos(S)$.

\section{Parameterizations}\label{S:Param}

The main purpose of this section is to study the general {\em change of variables} process within a local set theory. 

In subsection \ref{SS:Casogeneral}, a functor similar in spirit to the inverse image functor is considered and some of its basic properties are derived.  

In subsection \ref{SS:universalparam} the particular case of bijections from universal sets is analyzed and a complete parameterization is obtained that preserves all logical operations as well as the quantifiers (the latter with a minor natural correction term).  In particular, the results stated in \hyperlink{T:TeorB}{Theorem B} are proved herein. 
\subsection{Preimage translations}\label{SS:Casogeneral}

Let $S$ be a local set theory in a local language $\llocal$ and let $f$ be an arbitrary $S$-function from a set $X$ of type $\mathbf{PA}$ to a set $Y$ of type $\mathbf{PB}$.  Inspired by the natural pullback functor $f^*$ in the category of sets, one can define the {\em preimage formula translator}, $(-)^\natural_f:F\rightarrow F$, from the class $F$ of all formulae in $\llocal$.  

To regard $F$ as a category with arrows determined by the partial order $\vdash_S$ and have $(-)^\natural_f$ be a functor, one needs to consider the following axiom.
\begin{teorem} [Nullstellensatz]\label{A:Nullstellensatz} There is a function symbol of signature $\bf 1 \longrightarrow \typea$ for every ground type $\typea$.
\end{teorem}

The Lambek-Scott \cite{MR856915} deduction system doesn't need this Nullstellensatz to hold to have an unrestricted cut rule and therefore to have transitivity of deduction. This means that in that case $F$ is automatically a category.

In the presence of the Nullstellensatz, the derived rules \ref{t:universalelimination} and \ref{t:existentialintroduction} are valid without the freeness requirement.

\begin{observacion}\label{O:forallexistsadjoints} An easy immediate consequence of the Nullstellensatz, expressing in categorical terms the expected behavior of $\forall$ and $\exists$ (Cf. Theorem 1 of section 9 of Chapter I of \cite{MR1300636}) is the following: Let $F[x_1,\ldots,x_n]$ be the class of all formulae with free variables within $x_1,\ldots,x_n$. Then $\forall x_i$ and $\exists x_i$ are right and left adjoints of the functor $\varsigma:F[x_1,\ldots,\hat{x_i},\ldots,x_n]\hookrightarrow F[x_1,\ldots,x_n]$ defined as $\varsigma:=(-)\wedge x_i=x_i$:
\[\xymatrix{
F[x_1,\ldots,x_n]\ar@<2.5ex>[d]_\dashv^{\forall x_i}\ar@<-3.5ex>[d]^\dashv_{\exists x_i} \\
F[x_1,\ldots,\hat{x_i},\ldots,x_n].\ar[u]^\varsigma
}
\]
\end{observacion}

In \ref{T:starfisafunctor}, the corestriction  $(-)^\times_f:F\rightarrow F\langle y\rangle$ of $(-)^\natural_f$  to the class $F\langle y\rangle$ of all formulae in $\llocal$ not having $y$ as a free variable is shown to have a right adjoint. When regarded as a functor from $(F\downarrow y\in Y)$ to $(F\downarrow x\in X)$, it is reminiscent of Theorem 3 of section 9 of chapter I of \cite{MR1300636}; however, at this stage no assertion is made on the existence of a left adjoint. On the other hand, it is seen in \ref{T:starfipreserveswedgenvee} that $(-)^\natural_f$ is a morphism of lattices. 

\begin{definicion} \label{D:starfdef}Let $S$ be a local set theory on a local language $\llocal$. Let $f:X\rightarrow Y$ be an $S$-function with $X$ of type $\mathbf{PA}$ and $Y$ of type $\mathbf{PB}$, and $y$ a variable of type $\mathbf{B}$. Let $\bm{y}$ be a finite sequence of variables containing the free variables of $\vartheta$ and $y$. Denote by $\bm{x}$ the sequence of variables obtained from $\bm{y}$ by replacing $y$ with the variable $x$ of type $\typea$.  For any formula $\vartheta$ in $\llocal$, its {\em preimage translation by $f$}, or simply $f$-{\em translation}, is the formula
\begin{equation}\label{E:starfdef}
\vartheta_{\bm{y},f,x}^\natural :=(\llbracket\vartheta\rrbracket_{K,\bm{y}}\circ 1\times i_Y\circ 1\times f)^\natural(\bm{x}),
\end{equation}
where $(-)^\natural$ is given by \eqref{E:omegadirectimage}.
\end{definicion}
\begin{observacion}\label{O:slicingoverxinX}
Under the same hyphotheses of \ref{D:starfdef}, one has that
\[\vartheta_{\bm{y},f,x}^\natural\vdash_Sx\in X.\]
\end{observacion}

\begin{teorema}\label{T:starfisafunctor}
Let $S$ be a local set theory in a local language $\llocal$ with Nullstellensatz. Let $f:X\rightarrow Y$ be an $S$-function with $X$ of type $\mathbf{PA}$ and $Y$ of type $\mathbf{PB}$ and fix $x$ and $y$, variables of type $\typea$ and $\mathbf{B}$ respectively. Then, $(-)_{\bm{y},f,x}^\natural:F\rightarrow F$ is a functor, and its corestriction $(-)^\times_f$ to $F\langle y\rangle$, has  $\FE_f:\xi\mapsto \langle x,y\rangle\in\lvert f\rvert\Rightarrow\xi$ as a right adjoint. 
\end{teorema}

\begin{lema}\label{L:varthetastarfequalto}
Let $S$ be a local set theory in a local language $\llocal$. Let $f:X\rightarrow Y$ be an $S$-function with $X$ of type $\mathbf{PA}$ and $Y$ of type $\mathbf{PB}$. Then
\[\vdash_S\vartheta^\natural_{\bm{y},f,x}\Leftrightarrow\exists y(\langle x,y\rangle\in\lvert f\rvert\wedge\vartheta).\]
\end{lema}
\begin{proof}
Indeed, let $x_1,\ldots,x_n,y$ contain the variables of $\vartheta$. Let $x_1',\ldots,x_n'$ and $x_1'',\ldots,x_n''$ and  $x_1''',\ldots,x_n'''$ be variables of the same types as $x_1,\ldots,x_n$, respectively. Let $x'$ be of the same type as $x$, and $y',y'',y'''$ be different from but of the same type as $y$. Further assume that $y'$ is not free in $\vartheta$. Let $z,z'$ be variables of the same type as $\langle x_1,\ldots,x_n,y\rangle$. Then,
\begin{align*}
\langle\langle x_1,\ldots,x_n,x\rangle,z\rangle=\langle\langle x_1',\ldots,x_n',x'\rangle,\langle x_1',\ldots,x_n',y'\rangle\rangle\wedge& \\
x_1'\in A_1\wedge\cdots\wedge x_n'\in A_n\wedge\langle x',y'\rangle\in\lvert f\rvert\wedge& \\
\langle z,z'\rangle=\langle\langle x_1'',\ldots,x_n'',y''\rangle,\langle x_1'',\ldots,x_n'',y''\rangle\rangle\wedge& \\
y''\in Y\wedge x_1''\in A_1\wedge\cdots\wedge x_n''\in A_n\wedge& \\
\langle z',verus\rangle=\langle\langle x_1''',\ldots,x_n''',y'''\rangle,\vartheta\rangle\wedge& \\
x_1'''\in A_1\wedge\cdots\wedge x_n'''\in A_n\wedge y'''\in B\vdash_S& \\
\langle x,y'\rangle\in\lvert f\rvert\wedge\vartheta(y/y')=verus\vdash_S& \\
\langle x,y'\rangle\in\lvert f\rvert\wedge\vartheta(y/y')\vdash_S& \\
\exists y'(\langle x,y'\rangle\in\lvert f\rvert\wedge&\vartheta(y/y'));
\end{align*}

wherefrom, by \ref{t:existentializationontheleft},
\begin{align*}
\langle\langle x_1,\ldots,x_n,x\rangle,z\rangle=\langle\langle x_1',\ldots,x_n',x'\rangle,\langle x_1',\ldots,x_n',y'\rangle\rangle\wedge& \\
x_1'\in A_1\wedge\cdots\wedge x_n'\in A_n\wedge\langle x',y'\rangle\in\lvert f\rvert\wedge& \\
\langle z,z'\rangle=\langle\langle x_1'',\ldots,x_n'',y''\rangle,\langle x_1'',\ldots,x_n'',y''\rangle\rangle\wedge& \\
y''\in Y\wedge x_1''\in A_1\wedge\cdots\wedge x_n''\in A_n\wedge& \\
\exists x_1'''\cdots\exists x_n'''\exists y'''(\langle z',verus\rangle=\langle\langle x_1''',\ldots,x_n''',y'''\rangle,\vartheta\rangle\wedge &\\
x_1'''\in A_1\wedge\cdots\wedge x_n'''\in A_n\wedge y'''\in B)\vdash_S &\\
\exists y'(\langle x,y'\rangle\in\lvert f\rvert\wedge&\vartheta(y/y'));
\end{align*}
thus,
\begin{align*}
\langle\langle x_1,\ldots,x_n,x\rangle,z\rangle=\langle\langle x_1',\ldots,x_n',x'\rangle,\langle x_1',\ldots,x_n',y'\rangle\rangle\wedge& \\
x_1'\in A_1\wedge\cdots\wedge x_n'\in A_n\wedge\langle x',y'\rangle\in\lvert f\rvert\wedge& \\
\langle z,z'\rangle=\langle\langle x_1'',\ldots,x_n'',y''\rangle,\langle x_1'',\ldots,x_n'',y''\rangle\rangle\wedge& \\
y''\in Y\wedge x_1''\in A_1\wedge\cdots\wedge x_n''\in A_n\wedge& \\
\langle z',verus\rangle\in\lvert(\bm{y}\mapsto\vartheta)\rvert\vdash_S& \\
\exists y'(\langle x,y'\rangle\in\lvert f\rvert\wedge&\vartheta(y/y')).
\end{align*}
Following the same process, one obtains
\begin{multline}
\langle x_1,\ldots,x_n,x\rangle,z\rangle\in\lvert 1\times f\rvert\wedge\langle z,z'\rangle\in\lvert 1\times i_Y\rvert\wedge\langle z',verus\rangle\in\lvert(\bm{y}\mapsto\vartheta)\rvert\\
\vdash_S\exists y'(\langle x,y'\rangle\in\lvert f\rvert\wedge\vartheta(y/y')).\notag
\end{multline}
Lastly, by \ref{t:existentializationontheleft},
\begin{multline}
\langle\langle x_1,\ldots,x_n,x\rangle,verus\rangle\in\lvert(\bm{y}\mapsto\vartheta)\circ 1\times i_Y\circ 1\times f\rvert\vdash_S\\
\exists y'(\langle x,y'\rangle\in\lvert f\rvert\wedge\vartheta(y/y')).\notag
\end{multline}
Therefore,
\[\vartheta^\natural_{\bm{y},f,x}\vdash_S\exists y.\langle x,y\rangle\in\lvert f\rvert\wedge\vartheta.\]
Conversely,
\begin{align}
\langle x,y\rangle\in\lvert f\rvert\wedge\vartheta &\vdash_S\vartheta\notag\\
&\vdash_S\vartheta=verus\notag\\
&\vdash_S\langle\langle x_1,\ldots,x_n,y\rangle,verus\rangle=\langle\langle x_1,\ldots,x_n,y\rangle,\vartheta\rangle;\notag
\end{align}
wherefrom, by \ref{t:existsubstiright},
\begin{multline}
\langle x,y\rangle\in\lvert f\rvert\wedge\vartheta\vdash_S\\
\exists x_1'\cdots\exists x_n'\exists y'(\langle\langle x_1,\ldots,x_n,y\rangle,verus\rangle=\langle\langle x_1',\ldots,x_n',y'\rangle,\vartheta(\bm{y}/\bm{y'})\rangle),\notag
\end{multline}
so that
\[\langle x,y\rangle\in\lvert f\rvert\wedge\vartheta\vdash_S\langle\langle x_1,\ldots,x_n,y\rangle,verus\rangle\in\lvert(\bm{y}\mapsto\vartheta)\rvert.\]
Similarly,
\[\langle x,y\rangle\in\lvert f\rvert\wedge\vartheta\vdash_S\langle\langle x_1,\ldots,x_n,y\rangle,\langle x_1,\ldots,x_n,y\rangle\rangle\in\lvert 1\times i_Y\rvert\]
and
\[\langle x,y\rangle\in\lvert f\rvert\wedge\vartheta\vdash_S\langle\langle x_1,\ldots,x_n,x\rangle,\langle x_1,\ldots,x_n,y\rangle\rangle\in\lvert 1\times f\rvert.\]
Therefore,
\begin{multline}
\langle x,y\rangle\in\lvert f\rvert\wedge\vartheta\vdash_S\\
\langle\langle x_1,\ldots,x_n,y\rangle,verus\rangle\in\lvert(\bm{y}\mapsto\vartheta)\rvert\wedge\\
\langle\langle x_1,\ldots,x_n,y\rangle,\langle x_1,\ldots,x_n,y\rangle\rangle\in\lvert 1\times i_Y\rvert\wedge\\
\langle\langle x_1,\ldots,x_n,x\rangle,\langle x_1,\ldots,x_n,y\rangle\rangle\in\lvert 1\times f\rvert\notag.
\end{multline}
Thus, by \ref{t:existsubstiright},
\[\langle x,y\rangle\in\lvert f\rvert\wedge\vartheta\vdash_S\langle\langle x_1,\ldots,x_n,x\rangle,verus\rangle\in\lvert(\bm{y}\mapsto\vartheta)\circ 1\times i_Y\circ 1\times f\rvert.\]
This, together with \ref{t:existentializationontheleft}, yields that

\[ \exists y(\langle x,y\rangle\in\lvert f\rvert\wedge\vartheta)\vdash_S\vartheta^\natural_{\bm{y},f,x}.\qedhere
\]
\end{proof}

This lemma has a couple of immediate consequences. 
\begin{corolario}\label{C:superfluousvariablesforstar}
Let $S$ be an LST in a local language $\llocal$. Let $f:X\rightarrow Y$ be an $S$-function with $X$ of type $\mathbf{PA}$ and $Y$ of type $\mathbf{PB}$, and $\vartheta$ a formula in $\llocal$. Let $\bm{y}$ be a finite sequence of variables containing those of $\vartheta$, as well as $y$. Let $\bm{y'}$ be a subsequence of $\bm{y}$ still containing the free variables of $\vartheta$, as well as $y$. Then, $\vdash_S\vartheta_{\bm{y},f,x}^\natural=\vartheta_{\bm{y'},f,x}^\natural$.
\end{corolario}
\begin{notacion}
Under the same conditions, given that one can add superfluous variables, denote $\vartheta_{\bm{y},f,x}^\natural$ as $\vartheta^\natural_{y,f,x}$; if from context one can infer the fixed variables, as $\vartheta_f^\natural$; and if the $S$-function $f$ is clear from context, simply as $\vartheta^\natural$.
\end{notacion}

\begin{corolario}\label{C:starandconstants}
Let $S$ be an LST in a local language $\llocal$. Let $f:X\rightarrow Y$ be an $S$-function with $X$ of type $\mathbf{PA}$ determined by $\alpha$ and $Y$ of type $\mathbf{PB}$, and $\vartheta$ a formula in $\llocal$ not having $y$ as a free variable. Then, $\vdash_S\vartheta^\natural_f=\vartheta\wedge x\in X$.
\end{corolario}
\proof
From \ref{L:varthetastarfequalto} and \ref{t:slidingexistence}, one obtains $\vartheta^\natural_f\vdash_S\vartheta$ and, by \ref{O:slicingoverxinX},
\[\vartheta^\natural_f\vdash_S\vartheta\wedge x\in X.\]
Conversely,  it follows that $\vartheta\wedge x\in X\vdash_S\vartheta$ and
\[
x\in X \vdash_S\exists!y(\langle x,y\rangle\in\lvert f\rvert) \vdash_S\exists y(\langle x,y\rangle\in\lvert f\rvert).
\]
Wherefrom,
\belowdisplayskip=-12pt
\begin{align*}
\vartheta\wedge x\in X &\vdash_S\exists y(\langle x,y\rangle\in\lvert f\rvert)\wedge\vartheta\\
&\vdash_S\exists y(\langle x,y\rangle\in\lvert f\rvert\wedge\vartheta) &\text{by (\ref{t:slidingexistence})}\\
&\vdash_S\vartheta^\natural_f.
\end{align*}\qedhere
\endproof
\belowdisplayskip=12pt plus 3pt minus 9pt

\begin{proof}[Proof of \ref{T:starfisafunctor}] In light of \ref{t:unrestrictedcut}, the Nullstellensatz \ref{A:Nullstellensatz} yields the transitivity of deduction, which renders $F$ a category. 

Let $\vartheta,\gamma\in F$ be such that $\vartheta\vdash_S\gamma.$ Then,
\[
\langle x,y\rangle\in\lvert f\rvert\wedge\vartheta \vdash_S\langle x,y\rangle\in\lvert f\rvert\wedge\gamma\vdash_S\exists y(\langle x,y\rangle\in\lvert f\rvert\wedge\gamma).
\]
Therefore, by \ref{t:existentializationontheleft} and \ref{L:varthetastarfequalto},
\[\vartheta^\natural_f\vdash_S\gamma^\natural_f.\]

Now, to see that the corestriction $(-)^\natural_f$ to $F\langle y\rangle$ has $\xi\mapsto\langle x,y\rangle\in\lvert f\rvert\Rightarrow\xi$ as a right adjoint, first observe that $\xi\mapsto\langle x,y\rangle\in\lvert f\rvert\Rightarrow\xi$ is functorial. Indeed, let $\xi\vdash_S\zeta$. Thus, by thinning, $\xi,\langle x,y\rangle\in\lvert f\rvert\vdash_S\zeta$. So that by \ref{t:implicationinhypotheses} together with $\langle x,y\rangle\in\lvert f\rvert\vdash_S\langle x,y\rangle\in\lvert f\rvert$,
\[
\langle x,y\rangle\in\lvert f\rvert\Rightarrow\xi,\langle x,y\rangle\in\lvert f\rvert\vdash_S\zeta;
\]
wherefrom 
\[\langle x,y\rangle\in\lvert f\rvert\Rightarrow\xi\vdash_S\langle x,y\rangle\in\lvert f\rvert\Rightarrow\zeta.\]
Lastly, to see that $\xi\mapsto\langle x,y\rangle\in\lvert f\rvert\Rightarrow\xi$ is a right adjoint of $(-)^\times_f$, assume that $\vartheta^\times_f\vdash_S\xi$. Thus by \ref{L:varthetastarfequalto}, $\exists y(\langle x,y\rangle\in\lvert f\rvert\wedge\vartheta)\vdash_S\xi$. By \ref{t:existentialintroduction} and the cut rule,
\[
\langle x,y\rangle\in\lvert f\rvert\wedge\vartheta\vdash_S\xi.
\]
Wherefrom, $\vartheta\vdash_S\langle x,y\rangle\in\lvert f\rvert\Rightarrow\xi$.

Conversely, suppose $\vartheta\vdash_S\langle x,y\rangle\in\lvert f\rvert\Rightarrow\xi$. Thus, $\langle x,y\rangle\in\lvert f\rvert\wedge\vartheta\vdash_S\xi$. So that by \ref{t:existentializationontheleft},

\belowdisplayskip=-12pt
\[\exists y(\langle x,y\rangle\in\lvert f\rvert\wedge\vartheta)\vdash_S\xi.\qedhere
\]
\end{proof}
\belowdisplayskip=12pt plus 3pt minus 9pt

\begin{observacion}
The restriction of $(-)^\times_f$ to $(F\downarrow y\in Y)$ has
\[\xi\mapsto y\in Y\wedge(\langle x,y\rangle\in\lvert f\rvert\Rightarrow\xi)\]
as a right adjoint.
\end{observacion}

\begin{corolario}
Let $S$ be a local set theory in a local language $\llocal$ with Nullstellensatz. Let $\alpha$ be a formula in $\llocal$ and let $X:=\{\langle x_1,\ldots,x_n\rangle:\alpha\}$, with $x_1,\ldots,x_n$ variables of type $\typea_1,\ldots,\typea_n$, resp., containing the free variables of $\alpha$. Then, 
\[\vdash_S\FE_{T_X}(\xi)=(\alpha\Rightarrow\xi).
\]
\end{corolario}
\begin{proof}
Consider the following pullback diagram (cf. \ref{L:Xisalphatoo}):
\[
\xymatrix{
X\ar[r]\ar[d]_{i_X} & 1\ar[d]^\top \\
A_1\times\cdots\times A_n\ar[r]_(.7){(\bm{x}\mapsto\alpha)} & \subob.
}
\]
This implies that $(\bm{x}\mapsto\alpha)\circ i_X=T_X$. From which, $\vdash_S\langle\langle x_1,\ldots,x_n\rangle,\alpha\rangle\in\lvert T_X\rvert\Leftrightarrow\alpha$; indeed,
\begin{align}
\langle\langle x_1,\ldots,x_n\rangle,\alpha\rangle\in\lvert T_X\rvert &\vdash_S\langle x_1,\ldots,x_n\rangle\in X\notag\\
&\vdash_S\alpha\notag.
\end{align}
Conversely,
\belowdisplayskip=-12pt
\begin{align}
\alpha &\vdash_S\langle x_1,\ldots,x_n\rangle\in X\notag\\
&\vdash_S\langle\langle x_1,\ldots,x_n\rangle,\alpha\rangle\in\lvert(\bm{x}\mapsto\alpha)\rvert &\text{(Subsection \ref{SS:termsinducefunctions})}\notag\\
&\vdash_S\langle\langle x_1,\ldots,x_n\rangle,\alpha\rangle\in\lvert T_X\rvert\notag.
\end{align}\qedhere
\end{proof}
\belowdisplayskip=12pt plus 3pt minus 9pt

Recall that exponentiation in $F$ is given by $\alpha\wedge\vartheta\vdash_S\xi$ if and only if $\vartheta\vdash_S\alpha\Rightarrow\xi$. The following result exhibits $(-)^\times_f\dashv\FE_f$ as a generalization of exponentiation in $F$.
\begin{proposicion} Let $\alpha$ and $\vartheta$ be formulae, and let  $\omega$ be a variable of type $\bm{\Omega}$ such that $\alpha$ is free for $\omega$ in $\vartheta$ and $\omega$ is free in $\vartheta$. Then,

\[\vdash_S(\vartheta)^\times_{T_X}\Leftrightarrow\alpha\wedge\vartheta(\omega/\alpha).\]
\end{proposicion}
\proof
Notice that $\langle\langle x_1,\ldots,x_n\rangle,\omega\rangle\in\lvert T_X\rvert\vdash_S\omega=verus$ and that
\[
\langle\langle x_1,\ldots,x_n\rangle,\omega\rangle\in\lvert T_X\rvert\vdash_S\langle x_1,\ldots,x_n\rangle\in X\vdash_S\alpha\vdash_S\alpha=verus;
\]
from which, $\langle\langle x_1,\ldots,x_n\rangle,\omega\rangle\in\lvert T_X\rvert\wedge\vartheta\vdash_S\alpha\wedge\vartheta(\omega/\alpha)$, and, from this, by \ref{t:existentializationontheleft}, 
\[
\exists\omega(\langle\langle x_1,\ldots,x_n\rangle,\omega\rangle\in\lvert T_X\rvert\wedge\vartheta)\vdash_S\alpha\wedge\vartheta(\omega/\alpha).
\]

Conversely, $\alpha\wedge\vartheta(\omega/\alpha)\vdash_S\langle\langle x_1,\ldots,x_n\rangle,\alpha\rangle\in\lvert T_X\rvert\wedge\vartheta(\omega/\alpha)$; from which, by \ref{t:existsubstiright},
\[
\alpha\wedge\vartheta(\omega/\alpha)\vdash_S\exists\omega(\langle\langle x_1,\ldots,x_n\rangle,\omega\rangle\in\lvert T_X\rvert\wedge\vartheta),
\]
provided $\alpha$ is free for $\omega$ in $\vartheta$ and $\omega$ is free in $\vartheta$; So, under these conditions,
\belowdisplayskip=-12pt
\[\vdash_S\exists\omega(\langle\langle x_1,\ldots,x_n\rangle,\omega\rangle\in\lvert T_X\rvert\wedge\vartheta)\Leftrightarrow\alpha\wedge\vartheta(\omega/\alpha).\qedhere\]
\endproof
\belowdisplayskip=12pt plus 3pt minus 9pt

Finally, the following result continues the analysis of $(-)^\natural$ as a functor. It is a more general result than that in \ref{T:starpreservinglogicaloperations} below, albeit not as complete. 

\begin{teorema}\label{T:starfipreserveswedgenvee}
Given the same conditions of \ref{T:starfisafunctor}, it follows that $(-)^\natural_f$ is a morphism of lattices: For any $\vartheta$, $\gamma$, 
\begin{align}
&\vdash_S (\vartheta\wedge\gamma)^\natural_f=\vartheta^\natural_f\wedge\gamma^\natural_f,\notag\\
&\vdash_S (\vartheta\vee\gamma)^\natural_f =\vartheta^\natural_f\vee\gamma^\natural_f.\notag
\end{align}
\end{teorema}
\begin{proof}
From $\vartheta\wedge\gamma\vdash_S\vartheta$ and $\vartheta\wedge\gamma\vdash_S\gamma$, it follows that
\[(\vartheta\wedge\gamma)^\natural_f\vdash_S\vartheta^\natural_f\wedge\gamma^\natural_f.\]

Now, let $y,y'$ be variables such that $y'$ is free for $y$ in $\gamma$ and not free in $\gamma$, then
\[\langle x,y\rangle\in\lvert f\rvert\wedge\langle x,y'\rangle\in\lvert f\rvert\vdash_Sy'=y.\]
From this, by Equality, 
\[\langle x,y\rangle\in\lvert f\rvert\wedge\vartheta\wedge\langle x,y'\rangle\in\lvert f\rvert\wedge\gamma(y/y')\vdash_S\langle x,y\rangle\in\lvert f\rvert\wedge\vartheta\wedge\gamma.\]
Then, by \ref{t:existentialintroduction} and the cut rule,
\[\langle x,y\rangle\in\lvert f\rvert\wedge\vartheta\wedge\langle x,y'\rangle\in\lvert f\rvert\wedge\gamma(y/y')\vdash_S\exists y(\langle x,y\rangle\in\lvert f\rvert\wedge\vartheta\wedge\gamma).\]
Therefore, by \ref{t:existentializationontheleft},
\[\exists y(\langle x,y\rangle\in\lvert f\rvert\wedge\vartheta)\wedge\exists y'(\langle x,y'\rangle\in\lvert f\rvert\wedge\gamma(y/y'))\vdash_S\exists y(\langle x,y\rangle\in\lvert f\rvert\wedge\vartheta\wedge\gamma),\]
so that, by \ref{L:varthetastarfequalto},
\[\vartheta^\natural_f\wedge\gamma^\natural_f\vdash_S(\vartheta\wedge\gamma)^\natural_f.\]

In the case of $\vee$, from $\vartheta\vdash_S\vartheta\vee\gamma$ and $\gamma\vdash_S\vartheta\vee\gamma$, it follows that 
\[\vartheta^\natural_f\vee\gamma^\natural_f\vdash_S(\vartheta\vee\gamma)^\natural_f.\]

Now,
\begin{align}
\langle x,y\rangle\in\lvert f\rvert\wedge\vartheta &\vdash_S\exists y(\langle x,y\rangle\in\lvert f\rvert\wedge\vartheta)\notag\\
&\vdash_S\vartheta^\natural_f\notag\\
&\vdash_S\vartheta^\natural_f\vee\gamma^\natural_f.\notag
\end{align}
Similarly, $\langle x,y\rangle\in\lvert f\rvert\wedge\gamma\vdash_S\vartheta^\natural_f\vee\gamma^\natural_f$. Thus,
\[(\langle x,y\rangle\in\lvert f\rvert\wedge\vartheta)\vee(\langle x,y\rangle\in\lvert f\rvert\wedge\gamma)\vdash_S\vartheta^\natural_f\vee\gamma^\natural_f;\]
so that
\[\langle x,y\rangle\in\lvert f\rvert\wedge(\vartheta\vee\gamma)\vdash_S\vartheta^\natural_f\vee\gamma^\natural_f.\]
Therefore, by \ref{t:existentializationontheleft}  and \ref{L:varthetastarfequalto},
\[(\vartheta\vee\gamma)^\natural_f\vdash_S\vartheta^\natural_f\vee\gamma^\natural_f.\qedhere\]
\end{proof}

\begin{observacion}
Given the conditions in \ref{T:starfisafunctor}, it can be readily verified that
\[
(\vartheta\Rightarrow\gamma)^\natural_f\vdash_S\vartheta^\natural_f\Rightarrow\gamma^\natural_f\quad\text{and}\quad (\neg\vartheta)^\natural_f \vdash_S\neg\vartheta^\natural_f.
\]
\end{observacion}

\subsection{Universal parameterizations}\label{SS:universalparam}
The main purpose of this subsection is to prove the theorems in \ref{T:changingvariables} and \ref{T:starpreservinglogicaloperations}, which provide the logical framework for the set-theoretical parameterization of section \ref{s:representabilidad}. In particular, these parameterizations will be bijections from universal sets.

For a well-termed local set theory $S$ in a local language $\llocal$, there is functional representability from universals even if the codomain is not universal.

\begin{proposicion}\label{L:corestrictiontoX}
Let $S$ be a well-termed LST in a local language $\llocal$. Let $\typea$ and $\mathbf{B}$ be types. Let $f:U_\mathbf{B}\rightarrow X$ be an $S$-function with $X$ determined by a formula $\alpha$ with a unique free variable of type $\typea$. Then $f$ is the corestriction to $X$ of $(u\mapsto\tau_f)$ for some term $\tau_f$ in $\llocal$ and $\vdash_S\tau_f\in X$. In this case,  $\tau_f$ {\em represents $f$}.
\end{proposicion}
\begin{proof}
Consider the $S$-function $i_X\circ f$. Since $S$ is well-termed, $i_X\circ f=(u\mapsto\tau_f)$ for some term $\tau_f$. Thus, by \eqref{E:interpretationcorrespondstofunctions}, $\overbar{\llbracket\alpha\rrbracket_{K,x}}\circ f=\llbracket\tau_f\rrbracket_{K,u}$; wherefrom, by \eqref{E:relofmaximal}, 
\[\llbracket\alpha\rrbracket_{K,x}\circ\llbracket\tau_f\rrbracket_{K,u}=T_B.\]
This, together with \eqref{E:subst} and \ref{P:validityTrelation}, $\vDash_K\alpha(x/\tau_f)$. So that by \eqref{E:soundcompletenessforTS}, $\vdash_S\alpha(\tau_f)$; thus, 
\[\vdash_S\tau_f\in X.\]
Lastly, let $g:=(\lvert(u\mapsto\tau_f)\rvert,B,X)$. Since $i_X$ is monic and $i_X\circ g=(u\mapsto\tau_f)$, the conclusion follows.
\end{proof}

The following result provides the logical framework to the set-theoretical functional representability of Section \ref{s:representabilidad}.

\begin{teorema}\label{T:changingvariables}
Let $S$ be a well-termed local set theory in a local language $\llocal$. Let $\typea$  and $\mathbf{B}$ be types in $\llocal$. Let $f:U_{\mathbf{B}}\rightarrow X$ be a surjective $S$-function with $X$ determined by the formula $\alpha$ with unique variable of type $\typea$, and $\tau_f$ a term that represents $f$.  Let $\gamma$ and $\Gamma$ be a formula and a finite set of formulae in $\llocal$. Then, 
\[
\Gamma,x\in X\vdash_S\gamma\quad\text{ if and only if }\quad\Gamma(x/\tau_f)\vdash_S\gamma(x/\tau_f).
\]

Furthermore, if $f$ is injective, then, for every formula $\vartheta$, 
\[\vdash_S\vartheta=\vartheta^\natural_{f^{-1}}(x/\tau_f).\]
And in this case, 
\[
\Gamma\vdash_S\gamma\quad\text{if and only if}\quad\Gamma^\natural_{f^{-1}},x\in X\vdash_S\gamma^\natural_{f^{-1}}.
\]
\end{teorema}
\begin{observacion}
In the presence of the Nullstellensatz \ref{A:Nullstellensatz}, the functor $(-)^\natural_{f^{-1}}$ in this theorem shows something more than merely $\Sub(U_\mathbf{B})\cong\Sub(X)$, which is equivalent to $F[y]\cong(F[x]\downarrow x\in X)$. It shows that 
\[
F\cong(F\downarrow x\in X).
\]
On the other hand, as long as $S$ is consistent, $\topos(S)$ is not reducible to $F$, since $(F\downarrow x\in X)$ is never equivalent to $(\topos(S)\downarrow X)$ (See \ref{O:forallexistsadjoints}).
\end{observacion}

\begin{observacion} The operation $\xi \mapsto \xi(x/\tau)$ is functorial (like $(-)^\natural_{f^{-1}}$) provided one considers equivalence classes of formulae---i.e. in order to be able to use the substitution inference rule.
\end{observacion}

\begin{lema}\label{L:subsetdiag}
Let $S$ be a local set theory in a local language $\llocal$. Let $\xi,\gamma$ be formulae in $\llocal$. Let $x_1,\ldots,x_n$ be variables of types $\typea_1,\ldots,\typea_n$, containing the free variables of $\xi$ and $\gamma$. Let
\[Z:=\{\langle x_1,\ldots,x_n\rangle:\xi\}.\]
Then, $\xi\vdash_S\gamma$ if and only if the following diagram commutes in $\topos(S)$:
\[\xymatrix{
Z\ar[r]\ar[d]_{i_{Z}} & 1\ar[d]^\top \\
A_1\times\cdots\times A_n\ar[r]_(.7){(\bm{x}\mapsto\gamma)} & \subob.
}\]
\end{lema}
\begin{proof}
Let $W:=\{\langle x_1,\ldots,x_n\rangle:\gamma\}$. Then,
\begin{align}
\xi\vdash_S &\gamma\notag\\
&\text{iff}\quad \xi\vDash_K\gamma &\text{(\eqref{E:soundcompletenessforTS})}\notag\\
&\text{iff}\quad\llbracket\xi\rrbracket_{K,\bm{x}}\leq\llbracket\gamma\rrbracket_{K,\bm{x}} &\text{(\eqref{E:validsequent})}\notag\\
&\text{iff}\quad\text{there is an arrow } f:Z\rightarrow W\text{ such that }i_{Z}=i_{W}\circ f &\text{(\eqref{E:transferedorder})}\notag
\end{align}
which completes the proof.
\end{proof}
\begin{proof}[Proof of \ref{T:changingvariables}]
Without loss of generality, assume $\Gamma=\{\xi\}$ and that $u$, of type $\mathbf{B}$, is neither in  $\xi$ nor in $\gamma$. Replace an empty $\Gamma$ with $\{verus\}$.

Assume that $\xi,x\in X\vdash_S\gamma$. By substitution, 
\[\xi(x/\tau_f),\tau_f\in X\vdash_S\gamma(\tau_f).\]
By the cut rule with $\vdash_S\tau_f\in X$ ---by \ref{L:corestrictiontoX}---,
\[\xi(x/\tau_f)\vdash_S\gamma(\tau_f).\]

Conversely, assume $\xi(x/\tau_f)\vdash_S\gamma(\tau_f)$. Let $x_1,\ldots, x_n$ be variables of types $\typea_1,\ldots,\typea_n$, resp., containing the free variables of $\xi$ and $\alpha$, with $x_n=x$.  Let 
\[Z:=\{\langle x_1,\ldots,x_n\rangle:\xi\wedge\alpha\}\quad\text{and}\quad Z':=\{\langle x_1,\ldots,x_{n-1},u\rangle:\xi(x/\tau_f)\}.\]
Thus, the following are pullback diagrams:
\[
\xymatrix{
Z'\ar[rr]\ar[d]_{i_{Z'}} & & 1\ar[d]^\top &
Z\ar[rr]\ar[d]_{i_Z} & & 1\ar[d]^\top \\
A_1\times\cdots\times A_{n-1}\times B\ar[rr]_(.7){(\bm{x'}\mapsto\xi(x/\tau_f))} & & \subob &
A_1\times\cdots\times A_n\ar[rr]_(.6){(\bm{x}\mapsto\xi\wedge\alpha)} & & \subob,
}
\]
with $i_{Z'}=(\bm{x'}\mapsto\bm{x'})$ and $i_Z=(\bm{x}\mapsto\bm{x})$.

Now, since
\begin{align}
\langle x_1,\ldots,x_{n-1},x\rangle\in Z &\vdash_S\xi\wedge\alpha\vdash_S\alpha\vdash_Sx\in X\notag,
\end{align}
let $j_Z:=(\lvert i_Z\rvert,Z,A_1\times\cdots\times X)$. Then $i_Z=1\times i_X\circ j_Z$, and the diagram
\[\xymatrix{
Z\ar[rrr]\ar[d]_{j_Z} & & & 1\ar[d]^\top \\
A_1\times\cdots\times A_{n-1}\times X\ar[r]_(.6){1\times i_X} & A_1\times\cdots\times A_n\ar[rr]_(.63){(\bm{x}\mapsto\xi\wedge\alpha)} & & \subob
}\]
is a pullback diagram too. So is the following:
\[\xymatrix{
Z'\ar[rrrr]\ar[d]_{i_{Z'}} & & & & 1\ar[d]^\top \\
A_1\times\cdots\times B\ar[r]^{1\times f}\ar@/_1pc/[rr]_{1\times(u\mapsto\tau_f)}\ar@/_3pc/[rrrr]_{(\bm{x'}\mapsto\xi(x/\tau_f))} & A_1\times\cdots\times X\ar[r]^{1\times i_X} & A_1\times\cdots\times A_n\ar[rr]^(.63){(\bm{x}\mapsto\xi\wedge\alpha)} & & \subob.
}\]

Consider now the following diagram:

\[\xymatrix{
Z'\ar@/^2pc/[rrrr]\ar@{.>}[r]^(.6){f'}\ar[d]_{i_{Z'}} & Z\ar[d]^{j_Z}\ar[rrr] & & & 1\ar[d]^\top \\
A_1\times\cdots\times B\ar[r]_{1\times f} & A_1\times\cdots\times X\ar[r]_{1\times i_X} & A_1\times\cdots\times A_n\ar[rr]_(.63){(\bm{x}\mapsto\xi\wedge\alpha)} & & \subob.
}\]
Since the right square is a pullback and the exterior rectangle commutes, there is a unique $f':Z'\rightarrow Z$ making the diagram commute. Since the exterior diagram is a pullback, the left square is too. Now, since $f$ is epic, $f$ is surjective; thus, so is $1\times f$ and thus also epic. Therefore, $f'$ is epic since in a topos the pullback of an epic is epic.

Now consider the following diagram:

\[\xymatrix{
Z'\ar[r]^(.6){f'}\ar[d]_{i_{Z'}} & Z\ar[d]^{j_Z}\ar[rrr] & & & 1\ar[d]^\top \\
A_1\times\cdots\times B\ar[r]^{1\times f}\ar@/_1.5pc/[rrrr]_{(\bm{x'}\mapsto\gamma(x/\tau_f))} & A_1\times\cdots\times X\ar[r]^{1\times i_X} & A_1\times\cdots\times A_n\ar[rr]^(.63){(\bm{x}\mapsto\gamma)} & & \subob.
}\]
Since $\xi(x/\tau_f)\vdash_S\gamma(x/\tau_f)$, the external diagram commutes by \ref{L:subsetdiag}. Since the left square commutes, 
\[
(\bm{x}\mapsto\gamma)\circ i_Z\circ f'=\top\circ !_Z\circ f'.
\]
Since $f'$ is epic, the right square commutes; therefore by \ref{L:subsetdiag}, $\xi\wedge\alpha\vdash_S\gamma$ and thus  $\xi,x\in X\vdash_S\gamma$.

Lastly, suppose that $f$ is furthermore monic. Let $\vartheta$ be a formula the free variables of which are within $x_1,\ldots,x_{n-1},u$. Define
\begin{align}
\mathfrak{X} &:=\{\langle x_1,\ldots,x_{n-1},u\rangle:\vartheta\}\notag\\
N_{\mathfrak{X}} &:=\{\langle x_1,\ldots,x_{n-1},x\rangle:\vartheta^\natural\}.\notag
\end{align}
Recall that $\llbracket\vartheta\rrbracket_{K,x_1,\ldots,x_{n-1},u}=(\langle x_1,\ldots,x_{n-1},u\rangle\mapsto\vartheta)$ by \eqref{E:interpretationcorrespondstofunctions}, and that by \ref{P:omegafunctions}, $(\bm{x'}\mapsto\vartheta)\circ 1\times f^{-1}=(\bm{x}\mapsto\vartheta^\natural_{f^{-1}})$. Consider the following diagram:

\[\xymatrix{
N_{\mathfrak{X}}\ar@{-->}[r]^{g}\ar[d]_{j_{N_{\mathfrak{X}}}}\ar@/^1.5pc/[rr] & \mathfrak{X}\ar[r]\ar[d]^{i_{\mathfrak{X}}} & 1\ar[d]^\top \\
A_1\times\cdots\times X\ar[r]_(.52){1\times f^{-1}} & A_1\times\cdots\times B\ar[r]_(.7){(\bm{x'}\mapsto\vartheta)} & \subob.
}\]

Again, since the right square is a pullback and the exterior diagram commutes, there is a unique $g$ making the whole diagram commute. Since the rectangle is a pullback (the external diagram was a pullback already), the left square is too, so that 
\[
N_{\mathfrak{X}}\cong\mathfrak{X}.
\]

Now, considering the following commutative diagram,
\[\xymatrix{
\mathfrak{X}\ar@{-->}[r]\ar[d]_{i_{\mathfrak{X}}} & N_{\mathfrak{X}}\ar[r]^{g}\ar[d]_{j_{N_{\mathfrak{X}}}} & \mathfrak{X}\ar[r]\ar[d]^{i_{\mathfrak{X}}} & 1\ar[d]^\top \\
A_1\times\cdots\times B\ar[r]^(.52){1\times f} & A_1\times\cdots\times X\ar[r]^(.52){1\times f^{-1}}\ar@/_1.5pc/[rr]_{(\bm{x}\mapsto\vartheta^\natural)} & A_1\times\cdots\times B\ar[r]^(.7){(\bm{x'}\mapsto\vartheta)} & \subob,
}
\]
one clearly obtains $\vdash_S\vartheta^\natural(\tau_f)=\vartheta$, since $f$ is the corestriction to $X$ of $(u\mapsto\tau_f)$.  From here, $\xi\vdash_S\gamma$ if and only if $\xi^\natural,x\in X\vdash_S\gamma^\natural$, since $\xi^\natural(\tau_f)\vdash_S\gamma^\natural(\tau_f)$ if and only if $\xi^\natural,x\in X\vdash_S\gamma^\natural$, by the first part of the theorem.
\end{proof}

\begin{teorema}\label{T:starpreservinglogicaloperations}
Let  $S$ be a well-termed local set theory in a local language $\llocal$. Let $f:U_\mathbf{B}\rightarrow X$  be an $S$-bijection, with $X$ determined by a formula $\alpha$ with unique variable of type $\typea$. Let $\tau_f$ be a term representing $f$ and let $\vartheta$ and $\gamma$ be formulae in $\llocal$. Then, 
\[
\vdash_S(\vartheta\square\gamma)^\natural_{f^{-1}}=\vartheta^\natural_{f^{-1}}\square\gamma^\natural_{f^{-1}},
\]
where $\square$ is any binary logical connective. Furthermore, 
\[
\vdash_S(\neg\vartheta)^\natural_{f^{-1}}=\neg\vartheta^\natural_{f^{-1}},
\]
and if $u$ is free in $\vartheta$ but $x$ is not, then 
\[
\vdash_S\forall u\vartheta=\forall x\in X.\vartheta^\natural_{f^{-1}}\text{ and }\vdash_S\exists u\vartheta=\exists x\in X.\vartheta^\natural_{f^{-1}}.
\] 
\end{teorema}
\begin{proof}
By the previous theorem,
\[\vdash_S(\vartheta\square\gamma)^\natural(\tau_f)=\vartheta\square\gamma,\quad\vdash_S\vartheta^\natural(\tau_f)=\vartheta, \quad\vdash_S\gamma^\natural(\tau_f)=\gamma.\]
Wherefrom,
\[\vdash_S(\vartheta\square\gamma)^\natural(\tau_f)=\vartheta^\natural(\tau_f)\square\gamma^\natural(\tau_f).\]

Now, by \ref{L:corestrictiontoX}, $f=(\lvert(u\mapsto\tau_f)\rvert,B,X)$. So that
\begin{equation}\label{E:equalitythroughf}
(\bm{x}\mapsto(\vartheta\square\gamma)^\natural)\circ 1\times f=(\bm{x}\mapsto\vartheta^\natural\square\gamma^\natural)\circ 1\times f;
\end{equation}
thus since $f$ is iso, $1\times f$ is too. Therefore,
\[(\bm{x}\mapsto(\vartheta\square\gamma)^\natural)\text{ is equal to }(\bm{x}\mapsto\vartheta^\natural\square\gamma^\natural).\]
Notice that both functions go from $A_1\times\cdots\times X$ to $\subob$; however,
\begin{align}
N_{(\vartheta\square\gamma)^\natural} &:=\{\langle x_1,\ldots,x\rangle:(\vartheta\square\gamma)^\natural\},\notag\\
N_{\vartheta^\natural\square\gamma^\natural} &:=\{\langle x_1,\ldots,x\rangle:\vartheta^\natural\square\gamma^\natural\}\notag
\end{align}
are each the pullback object of the $S$-functions $(\bm{x}\mapsto(\vartheta\square\gamma)^\natural)$ and $(\bm{x}\mapsto\vartheta^\natural\square\gamma^\natural)$ from $A_1\times\cdots\times A$ to $\subob$, so that 
\[\vdash_S(\vartheta\square\gamma)^\natural=\vartheta^\natural\square\gamma^\natural.\]
Likewise, $\vdash_S(\neg\vartheta)^\natural=\neg\vartheta^\natural$.

Lastly, assume that $u$ is free in $\vartheta$ and $x$ is not.  Then, by \ref{t:universalelimination}, $\forall u\vartheta\vdash_S\vartheta$. Thus, by \ref{T:changingvariables} and \ref{O:slicingoverxinX}, $(\forall u\vartheta)^\natural\vdash_S\vartheta^\natural$; so that by \ref{C:starandconstants} $\vdash_S(\forall u\vartheta)^\natural=\forall u\vartheta\wedge x\in X$. Therefore, 
\[
\forall u\vartheta\wedge x\in X\vdash_S\vartheta^\natural.
\]
From here, $\forall u\vartheta\vdash_Sx\in X\Rightarrow\vartheta^\natural$, and thus, by \ref{t:universalizationontheright},
\[
\forall u\vartheta\vdash_S\forall x\in X.\vartheta^\natural.
\]

Conversely,  by \ref{t:universalelimination} $\forall x\in X.\vartheta^\natural\vdash_Sx\in X\Rightarrow\vartheta^\natural$. From here, $x\in X,\forall x\in X.\vartheta^\natural\vdash_S\vartheta^\natural$. Thus by \ref{T:changingvariables},
\[
(\forall x\in X.\vartheta^\natural)(\tau_f)\vdash_S\vartheta.
\]
Since $u$ is not present in $\forall x\in X.\vartheta^\natural$, it follows, by \ref{t:universalizationontheright}, that
\[
\forall x\in X.\vartheta^\natural\vdash_S\forall u\vartheta,\]
and thus
\[
\vdash_S\forall u\vartheta=\forall x\in X.\vartheta^\natural
\]

Finally, by \ref{t:existentialintroduction}, $\vartheta\vdash_S\exists u\vartheta$. Thus
\begin{align}
x\in X,\vartheta^\natural &\vdash_S(\exists u\vartheta)^\natural &\text{(\ref{T:changingvariables})}\notag\\
&\vdash_S\exists u\vartheta\wedge x\in X &\text{(\ref{C:starandconstants})}\notag\\
&\vdash_S\exists u\vartheta\notag.
\end{align}
Therefore, by \ref{t:existentializationontheleft},
\[
\exists x\in X.\vartheta^\natural\vdash_S\exists u\vartheta.
\]

Conversely, by \ref{t:existentialintroduction}, $x\in X\wedge\vartheta^\natural\vdash_S\exists x\in X.\vartheta^\natural$. Now, by \ref{T:changingvariables}, $\vartheta\vdash_S(\exists x\in X.\vartheta^\natural)(\tau_f)$. Since $u$ is not present in $\exists x\in X.\vartheta^\natural$, by \ref{t:existentializationontheleft}, it follows that
\[
\exists u\vartheta\vdash_S\exists x\in X.\vartheta^\natural.
\]
and
\[
\vdash_S\exists u\vartheta=\exists x\in X.\vartheta^\natural\qedhere
\]
\end{proof}

Interpreting this in the language of Lindenbaum-Tarski algebras (Cf. Chapter VI Section 10 \cite{MR0163850}) produces the following result. 

\begin{corolario} Theorem \ref{T:starpreservinglogicaloperations} together with the Nullstellensatz \ref{A:Nullstellensatz} renders $\vartheta^\natural$ a morphism of Lindenbaum-Tarski algebras. Furthermore, restricting and co-restricting,
\[\xymatrix{
F[y]\ar[r]_(.34){\cong}^(.35){(-)^\natural} & (F[x]\downarrow x\in X).
}\]
\end{corolario}
This result is in agreement with the remark in \cite[p. 174]{MR856915} saying that in some sense $\topos(S)$ is a Lindenbaum-Tarski category.

\section{Functional representability}\label{s:representabilidad}
The main purpose of this section is to provide the set theoretical counterpart to the results in Section \ref{S:Param}. The results of \hyperlink{T:TeorA}{Theorem A} are proved in Subsection \ref{SS:param}.  In Subsection \ref{SS:funcrep}, the final step in achieving the compatible representation promised by \hyperlink{T:TeorC}{Theorem C} is completed. 

\subsection{Canonical parameterization}\label{SS:param}
The following theorem is the main result of the section: given any $S$-function $f$---represented within $\teo(\topos(S))$ via the canonical translation---, there is a canonical parameterization inside $\teo(\topos(S))$ under which $f$ is represented by itself. 
\begin{teorema}\label{T:Unicidaddeefe}Let $S$ be a local set theory in a local language  $\llocal$. Let $f:X\rightarrow Y$ be an $S$-function with $X$ of type $\mathbf{PA}$ and $Y$ of type $\mathbf{PB}$. Let $i_X$ y $i_Y$ be the functions $(x\mapsto x):X\rightarrow U_{\typea}$ and  $(y\mapsto y):Y\rightarrow U_{\bm B}$, respectively.
Then, 
\begin{equation}\label{E:symboloffasfunction}
\vdash_{\teo(\topos(S))}\langle\bm{i_X}(u),\bm{i_Y}(\typef(u))\rangle\in\lvert f\rvert.
\end{equation}
\end{teorema}
\begin{observacion}
Since $\vdash_{\teo(\topos(S))}\exists!v.\langle\bm{i_X}(u),\bm{i_Y}(v) \rangle\in \lvert f\rvert$, by \ref{T:eliminabilityinTeo} $f$ is the unique $S$-function that satisfies \eqref{E:symboloffasfunction}.
\end{observacion}
The following results will be needed in the proof of \ref{T:Unicidaddeefe}.  Their conclusions are the expected. They touch on the compatibility between the canonical translation of a local set theory into that of its linguistic topos and the canonical inclusion of the latter (See \hyperref[A:inter]{Appendix}). 
\begin{lema}\label{L:belongingequalsamonic}
Let $\E$ be a topos and $\alpha$ be a formula in $\mathrm{Th}(\E)$ with a single free variable. Then,
\[\vdash_{\mathrm{Th}(\E)}x\bm{\in\{}x\bm{:}\alpha\bm{\}}\Leftrightarrow\exists u.x\bm{=i}(u),\]
where $i$ is $\overbar{\llbracket\alpha\rrbracket_x}$.
\end{lema}
\begin{proof}
Since $i=\overbar{\llbracket\alpha\rrbracket_x}$, it follows that $\llbracket\alpha\rrbracket_x=\chi(i)$. From which,
\begin{align}
\llbracket\bm{\chi(i)}(x)\rrbracket_x &=\chi(i)\circ\llbracket x\rrbracket_x\notag\\
&=\chi(i)\notag\\
&=\llbracket\alpha\rrbracket_x\notag\\
&=\llbracket x\bm{\in\{}x\bm{:}\alpha\bm{\}}\rrbracket_x &\text{(\eqref{E:TeosiiE})};\notag
\end{align}
wherefrom, by \ref{P:equalityofsymbols},
\[\vDash_{H_\E}\bm{\chi(i)}(x)\Leftrightarrow x\bm{\in\{}x\bm{:}\alpha\bm{\}}.\]
Then, $\vdash_{\mathrm{Th}(\E)}\bm{\chi(i)}(x)\Leftrightarrow x\bm{\in\{}x\bm{:}\alpha\bm{\}}$, and, by \ref{L:substintoposEbis},
\[  \vdash_{\mathrm{Th}(\E)}\exists u.x\bm{=i}(u)\Leftrightarrow x\bm{\in\{}x\bm{:}\alpha\bm{\}}.\qedhere\]
\end{proof}

\begin{lema}\label{L:Ssetisomorphictoitssymboltype}
Let $S$ be a local set theory in a local language $\llocal$.  Let $X$ be an $S$-set determined by the formula $\alpha$. Then, 
\begin{equation}\label{E:alphaiu}
\vdash_{\teo(\topos(S))}\alpha(x/\bm{i_X}(u)),
\end{equation}
with $u$ a variable of type $\bm{X}$; and thus 
\begin{equation}\label{E:iu}
\vdash_{\teo(\topos(S))}\bm{i_X}(u)\bm{\in} X.
\end{equation}
Furthermore,  one has that the arrow $r_X:U_{\bm{X}}\rightarrow X$, the corestriction of $(u\mapsto\bm{i_X}(u))$ to $X$,  is an isomorphism.  (Recall that $\bm{X}$ is the type symbol associated with $X$ in $\llocal(\topos(S))$ and that $\bm{i_X}$ is a function symbol not present in $\llocal$).
\end{lema}
\begin{proof}
By \ref{L:Xisalphatoo}, the diagram \eqref{E:pullbackalpha} is a pullback diagram in $\topos(S)$, and hence in $\topos(\teo(\topos(S)))$ (where $x$ is a variable of type $\typea$, since the following translation is being used: $\eta_Sx_\typea=x_{\eta_S\typea}=x_\typea$, and $\llbracket\alpha\rrbracket_{K_S,x}=\llbracket\alpha\rrbracket_{H_{\topos(S)},x}$, by \eqref{E:canonical-natural}).

Now, since
\[\overbar{\llbracket\alpha\rrbracket_{H_{\topos(S)},x}}\circ r=i_X=\llbracket\bm{i_X}(u)\rrbracket_{H_{\topos(S)},u}\]
for some isomorphism $r$ in $\topos(S)$, \eqref{E:alphaiu} and \eqref{E:iu} follow from \ref{L:substintoposE}.
Lastly, to see that $r_X$ is an isomorphism, define $\lvert h\rvert$ as the $\teo(\topos(S))$-set 
\[\bm{\{}\langle x,u\rangle\bm{:i_X}(u)\bm{=}x\bm{\}}.\]
As the notation suggests,  $h$ is a function and, thus it is clearly the inverse of $r_X$. Indeed, by \ref{L:belongingequalsamonic} and the fact that $i_X$ is monic,
\[\vdash_{\teo(\topos(S))}x\bm{\in\{}x\bm{:}\alpha\bm{\}}\Leftrightarrow\exists! u.x\bm{=i_X}(u);\]
wherefrom,
\[\vdash_{\teo(\topos(S))}x\bm{\in}X\Rightarrow\exists! u.x\bm{=i_X}(u),\]
and thus $h$ is a $\teo(\topos(S))$-function $X\bm{\rightarrow}U_{\bm{X}}$. 
\end{proof}
\begin{proof}[Proof of \ref{T:Unicidaddeefe}]
Under the canonical translation $\eta_S$, one can regard $S$ as a subtheory of $\teo(\topos(S))$ and the associated logical functor $\topos(\eta_S)$ as the inclusion. 

Consider the following diagram in  $\topos(\teo(\topos(S)))$:
\begin{equation}\label{D:fasafunctionsymbol}
\xymatrix{
X\ar[rrr]^f\ar[dd]_{(x\mapsto x)} & & & Y\ar[d]^{(y\mapsto y)} \\
& & & B\ar[d]^{(z\mapsto\bm{\{}z\bm{\}})} \\
A\ar[rrr]_{(x\mapsto\bm{\{}y\bm{:}\langle x,y\rangle\in\lvert f\rvert\bm{\}})} & & & PB.
}
\end{equation}
It is a commutative diagram; in other words, 
\[
\vdash_{\teo(\topos(S))}\lvert(x\mapsto\bm{\{}y\bm{:}\langle x,y\rangle\in\lvert f\rvert\bm{\}})\rvert\bm{=}\lvert(y\mapsto\bm{\{}y\bm{\}})\circ f\rvert.
\]
To see this, first observe that
\begin{align}
\langle x,z\rangle\in\lvert f\rvert, u\in\bm{\{}z\bm{:}\langle x,z\rangle\in\lvert f\rvert\bm{\}} &\vdash_{\teo(\topos(S))}\langle x,z\rangle\in\lvert f\rvert\wedge\langle x,u\rangle\in\lvert f\rvert\notag\\
&\vdash_{\teo(\topos(S))} u\bm{=}z\qquad\qquad\qquad\text{(\ref{P:functionimpliesy=z})}\notag\\
&\vdash_{\teo(\topos(S))} u\in\bm{\{}z\bm{\}}\notag,
\end{align}
and that
\begin{align}
\langle x,z\rangle\in\lvert f\rvert,u\in\bm{\{}z\bm{\}} &\vdash_{\teo(\topos(S))}\langle x,z\rangle\in\lvert f\rvert\wedge u\bm{=}z\notag\\
&\vdash_{\teo(\topos(S))}\langle x,u\rangle\in\lvert f\rvert\notag\\
&\vdash_{\teo(\topos(S))} u\in\bm{\{}z\bm{:}\langle x,z\rangle\in\lvert f\rvert\bm{\}}\notag;
\end{align}
so that, by the equivalence and extensionality rules,
\begin{equation}\label{E:singletonfgraph}
\langle x,z\rangle\in\lvert f\rvert\vdash_{\teo(\topos(S))}\bm{\{}z\bm{:}\langle x,z\rangle\in\lvert f\rvert\bm{\}=\{}z\bm{\}}.
\end{equation}

Since
\begin{equation}\label{E:omegasingleton}
\langle z,b\rangle\in\lvert(y\mapsto\bm{\{}y\bm{\}})\rvert\vdash_{\teo(\topos(S))}b\bm{=\{}z\bm{\}},
\end{equation}
it follows from \eqref{E:omegasingleton} and \eqref{E:singletonfgraph} that
\[\langle x,z\rangle\in\lvert f\rvert\wedge\langle z,b\rangle\in\lvert(y\mapsto\bm{\{}y\bm{\}})\rvert\vdash_{\teo(\topos(S))}b\bm{=\{}z\bm{:}\langle x,z\rangle\in\lvert f\rvert\bm{\}};\]
wherefrom, by \ref{t:existentializationontheleft},
\[\exists z(\langle x,z\rangle\in\lvert f\rvert\wedge\langle z,b\rangle\in\lvert(y\mapsto\bm{\{}y\bm{\}})\rvert)\vdash_{\teo(\topos(S))}b\bm{=\{}z\bm{:}\langle x,z\rangle\in\lvert f\rvert\bm{\}}.
\]

By \eqref{E:compfun} and the cut rule,
\[
\langle x,b\rangle\in\lvert(y\mapsto\bm{\{}y\bm{\}})\circ f\rvert\vdash_{\teo(\topos(S))}\langle x,b\rangle\bm{=}\langle x,\bm{\{}z\bm{:}\langle x,z\rangle\in\lvert f\rvert\bm{\}}\rangle.
\]
So that, by
\begin{align}
\langle x,b\rangle\in\lvert(y &\mapsto\bm{\{}y\bm{\}})\circ f\rvert\vdash_{\teo(\topos(S))} x\in X\notag\\
&\vdash_{\teo(\topos(S))}\langle x,\bm{\{}z\bm{:}\langle x,z\rangle\in\lvert f\rvert\bm{\}}\rangle\in\lvert(x\mapsto\bm{\{}z\bm{:}\langle x,z\rangle\in\lvert f\rvert\bm{\}})\rvert\notag,
\end{align}
it follows that
\begin{equation}\label{E:part1}
\langle x,b\rangle\in\lvert(y\mapsto\bm{\{}y\bm{\}})\circ f\rvert\vdash_{\teo(\topos(S))}\langle x,b\rangle\in\lvert(x\mapsto\bm{\{}z\bm{:}\langle x,z\rangle\in\lvert f\rvert\bm{\}})\rvert.
\end{equation}
To see the converse of \eqref{E:part1}, notice that
\[
\langle x,z\rangle\in\lvert f\rvert\vdash_{\teo(\topos(S))}z\in Y\vdash_{\teo(\topos(S))}\langle z,\bm{\{}z\bm{\}}\rangle\in\lvert(y\mapsto\bm{\{}y\bm{\}})\rvert
\]
yields
\begin{equation}\label{E:fgivesfcomposedwith{y}}
\langle x,z\rangle\in\lvert f\rvert\vdash_{\teo(\topos(S))}\langle x,z\rangle\in\lvert f\rvert\wedge\langle z,\bm{\{}z\bm{\}}\rangle\in\lvert(y\mapsto\bm{\{}y\bm{\}})\rvert.
\end{equation}
Now, since
\[
\langle x,b\rangle\in\lvert(x\mapsto\bm{\{}z\bm{:}\langle x,z\rangle\in\lvert f\rvert\bm{\}})\rvert\vdash_{\teo(\topos(S))}b\bm{=\{}z\bm{:}\langle x,z\rangle\in\lvert f\rvert\bm{\}},
\]
it follows from \eqref{E:singletonfgraph} and \eqref{E:fgivesfcomposedwith{y}} that
\begin{multline}
\langle x,b\rangle\in\lvert(x\mapsto\bm{\{}z\bm{:}\langle x,z\rangle\in\lvert f\rvert\bm{\}})\rvert,\langle x,z\rangle\in\lvert f\rvert\vdash_{\teo(\topos(S))}\\
\langle x,z\rangle\in\lvert f\rvert\wedge\langle z,b\rangle\in\lvert(y\mapsto\bm{\{}y\bm{\}})\rvert;\notag
\end{multline}
wherefrom, by \ref{t:existentialintroduction},
\begin{multline}
\langle x,b\rangle\in\lvert(x\mapsto\bm{\{}z\bm{:}\langle x,z\rangle\in\lvert f\rvert\bm{\}})\rvert,\langle x,z\rangle\in\lvert f\rvert\vdash_{\teo(\topos(S))}\\
\exists z(\langle x,z\rangle\in\lvert f\rvert\wedge\langle z,b\rangle\in\lvert(y\mapsto\bm{\{}y\bm{\}})\rvert),\notag
\end{multline}
and thus
\begin{multline}
\langle x,b\rangle\in\lvert(x\mapsto\bm{\{}z\bm{:}\langle x,z\rangle\in\lvert f\rvert\bm{\}})\rvert,\langle x,z\rangle\in\lvert f\rvert\vdash_{\teo(\topos(S))}\\
\langle x,b\rangle\in\lvert(y\mapsto\bm{\{}y\bm{\}})\circ f\rvert,\notag
\end{multline}
so that, by \ref{t:existentializationontheleft},
\begin{multline}
\langle x,b\rangle\in\lvert(x\mapsto\bm{\{}z\bm{:}\langle x,z\rangle\in\lvert f\rvert\bm{\}})\rvert,\exists z(\langle x,z\rangle\in\lvert f\rvert)\vdash_{\teo(\topos(S))}\\
\langle x,b\rangle\in\lvert(y\mapsto\bm{\{}y\bm{\}})\circ f\rvert.\notag
\end{multline}
However,
\begin{align}
\langle x,b\rangle\in\lvert(x\mapsto\bm{\{}z\bm{:}\langle x,z\rangle\in\lvert f\rvert\bm{\}})\rvert &\vdash_{\teo(\topos(S))} x\in X\notag\\
&\vdash_{\teo(\topos(S))}\exists! z(\langle x,z\rangle\in\lvert f\rvert)\notag\\
&\vdash_{\teo(\topos(S))}\exists z(\langle x,z\rangle\in\lvert f\rvert),\notag
\end{align}
and thus
\begin{equation}\label{E:part2}
\langle x,b\rangle\in\lvert(x\mapsto\bm{\{}z\bm{:}\langle x,z\rangle\in\lvert f\rvert\bm{\}})\rvert\vdash_{\teo(\topos(S))}\langle x,b\rangle\in\lvert(y\mapsto\bm{\{}y\bm{\}})\circ f\rvert.
\end{equation}
Now, by the equivalence rule applied to \eqref{E:part1} and \eqref{E:part2}, it follows from \eqref{E:equalfunctions} that \eqref{D:fasafunctionsymbol} commutes.

Under the natural interpretation of $\teo(\topos(S))$ in $\topos(S)$---denoted by $\llbracket\tau\rrbracket_{\bm{x}}=\llbracket\tau\rrbracket_{H_{\topos(S)},\bm{x}}$, where $\tau$ is a term of $\teo(\topos(S))$---it follows from \eqref{E:interpretationcorrespondstofunctions} and \eqref{E:canonical-natural} that the following diagram is the same as \eqref{D:fasafunctionsymbol} and that it also commutes:
\[
\xymatrix{
X\ar[rrr]^f\ar[d]_1 & & & Y\ar[d]^{\llbracket\bm{\{i_Y}(v)\bm{\}}\rrbracket_v} \\
X\ar[rrr]^{\llbracket\bm{\{}y\bm{:}\langle\bm{i_X}(u),y\rangle\in\lvert f\rvert\bm{\}}\rrbracket_u}\ar[dr]_{\llbracket\bm{i_X}(u)\rrbracket_u} & & & PB;\\
& A\ar[urr]|(.371)\hole_{\llbracket\bm{\{}y\bm{:}\langle x,y\rangle\in\lvert f\rvert\bm{\}}\rrbracket_x.} & &
}
\]

Since $\llbracket \typef(u)\rrbracket_u=f$, it follows that 
\[
\llbracket\bm{\{}y\bm{:}\langle\bm{i_X}(u),y\rangle\in\lvert f\rvert\bm{\}}\rrbracket_u=\llbracket\bm{\{i_Y}(\typef(u))\bm{\}}\rrbracket_u;
\]
and, by \ref{P:equalityofsymbols},
\[\vDash_{H_{\topos(S)}}\bm{\{}y\bm{:}\langle\bm{i_X}(u),y\rangle\in\lvert f\rvert\bm{\}=\{i_Y}(\typef(u))\bm{\}}.\]
Therefore, by \eqref{E:TeosiiE},
\[\vdash_{\teo(\topos(S))}\bm{\{}y\bm{:}\langle\bm{i_X}(u),y\rangle\in\lvert f\rvert\bm{\}=\{i_Y}(\typef(u))\bm{\}}.\]
Thus, since $\bm{\{i_Y}(\typef(u))\bm{\}}$ is by definition $\bm{\{}y\bm{:}y\bm{=i_Y}(\typef(u))\bm{\}}$, by \ref{t:samesetsamefourmulas},
\begin{equation}\label{E:y=fu}
\vdash_{\teo(\topos(S))}y\bm{=i_Y}(\typef(u))\Leftrightarrow\langle\bm{i_X}(u),y\rangle\in\lvert f\rvert.
\end{equation}
From this,
\begin{align}
\langle\bm{i_X}(u),y\rangle\in\lvert f\rvert &\vdash_{\teo(\topos(S))}y\bm{=i_Y}(\typef(u))\wedge\langle\bm{i_X}(u),y\rangle\in\lvert f\rvert\notag\\
&\vdash_{\teo(\topos(S))}\langle\bm{i_X}(u),\bm{i_Y}(\typef(u))\rangle\in\lvert f\rvert\notag.
\end{align}

So that, by \ref{t:existentializationontheleft},
\[\exists y(\langle\bm{i_X}(u),y\rangle\in\lvert f\rvert)\vdash_{\teo(\topos(S))}\langle\bm{i_X}(u),\bm{i_Y}(\typef(u))\rangle\in\lvert f\rvert.\]
On the other hand, by \eqref{E:iu},
\begin{align}
\vdash_{\teo(\topos(S))}\bm{i_X}(u)\in X &\vdash_{\teo(\topos(S))}\exists!y(\langle\bm{i_X}(u),y\rangle\in\lvert f\rvert)\notag\\
&\vdash_{\teo(\topos(S))}\exists y(\langle\bm{i_X}(u),y\rangle\in\lvert f\rvert)\notag.
\end{align}
wherefrom, by the cut rule, \eqref{E:symboloffasfunction} follows.
\end{proof}
\begin{corolario}\label{P:functionisotoitsrepresentation}
Let $S$ be a local set theory in a local language $\llocal$. Let $f:X\rightarrow Y$ be an $S$-function with $X$ of type $\mathbf{PA}$ and $Y$ of type $\mathbf{PB}$. Then, $f$ is the only $S$-function for which the following diagram commutes in $\topos(\teo(\topos(S)))$:
\begin{equation}\label{E:diagramanatural}
\xymatrix{
X\ar[r]^\cong\ar[d]_f & U_{\bm{X}}\ar[d]^{(x\mapsto\typef(x))} \\
Y\ar[r]_\cong & U_{\bm{Y}}.
}
\end{equation}
In other words, the inclusion functor $\iota_{\topos(S)}$ and the logical functor induced by the canonical translation $\eta_S$ are isomorphic:
\[\topos(\eta_S)\cong \iota_{\topos(S)}:\topos(S)\rightarrow\topos(\teo(\topos(S))).\]
\end{corolario}
\begin{proof}
Equation \eqref{E:symboloffasfunction} is equivalent to the commutativity of diagram \eqref{E:diagramanatural}.
\end{proof}

\subsection{Strict representation}\label{SS:funcrep}
All that remains to verify is that the arrow $(u\mapsto \typef(u))$ corresponds to $f$ as syntactic functions under the change of variables of  \ref{T:changingvariables}. To this effect, the following result gives an explicit description of the required translation of $f$.
\begin{lema}\label{L:naturalf} Let $S$ be a local set theory in a local language $\llocal$. Let $f:X\rightarrow Y$ be an $S$-function, and let $\gamma$ be a formula that determines $\lvert f\rvert$. Then 
\begin{equation}\label{E:fstar1} \vdash_{Th(\topos(S))}(\gamma(x/\bm{i_X}(u),y/\bm{i_Y}(v)))^\natural_{r_X^{-1}\times r_Y^{-1}}=\gamma.
\end{equation}
Furthermore, 
\[
(\{\langle u,v\rangle:\gamma(x/\bm{i_X}(u),y/\bm{i_Y}(v))\}, U_{\bf X}, U_{\bf Y} )
\]
is a $\teo(\topos(S))$-function henceforth denoted by  $f^*$.  
\end{lema} 
\proof Denote by $\lvert f^*\rvert$ the set $\{\langle u,v\rangle:\gamma(x/\bm{i_X}(u),y/\bm{i_Y}(v))\}$. The equation \eqref{E:fstar1} follows from the fact that the following is a pullback of pullbacks:
\[\xymatrix{
\lvert f^*\rvert\ar[rr]^\cong\ar[d]_{i_{\lvert f^*\rvert}} & & \lvert f\rvert\ar[rrr]\ar[d]_{j_{\lvert f\rvert}} & & & 1\ar[d]^\top \\
U_\mathbf{X}\times U_\mathbf{Y}\ar[rr]^{r_X\times r_Y}_\cong\ar@/_1.5pc/[rrr]_(.65){(\langle u,v\rangle\mapsto\langle\bm{i}_X(u),\bm{i}_Y(v)\rangle)}\ar@/_3pc/[rrrrr]_{(\langle u,v\rangle\mapsto\gamma(\bm{i}_X(u),\bm{i}_Y(v)))} & & X\times Y\ar[r]^{i_X\times i_Y} & A\times B\ar[rr]^(.55){(\langle x,y\rangle\mapsto\gamma)} & & \subob,
}\]
the proof of which is {\em mutatis mutandis} already within the proof of \ref{T:changingvariables}. A direct application of \ref{T:starpreservinglogicaloperations} proves that  $f^*$ is a function.
\endproof
\begin{teorema}\label{T:TeoA}
Let $S$ be a local set theory in a local language $\llocal$. Let $X$ and $Y$ be $S$-sets, with $\alpha$ and $\beta$ the formulae determining them, resp., and let $f:X\rightarrow Y$ be an $S$-function with $\gamma$ the formula that determines $\lvert f\rvert$. Then,
\begin{equation}\label{E:ufuiniotaf}
\vdash_{Th(\topos(S))} \langle u,\typef(u) \rangle\in \lvert f^*\rvert.
\end{equation}
In other words, that $f^*=(u\mapsto \typef(u))$.
\end{teorema}
\begin{proof}
It is immediate that
\[
\langle\bm{i_X}(u),\bm{i_Y}(\typef(u))\rangle\in\lvert f\rvert \vdash_{\teo(\topos(S))}\langle u,\typef(u)\rangle\in\lvert f^*\rvert\notag,
\]
thus \eqref{E:ufuiniotaf} follows by \ref{T:Unicidaddeefe}. Finally, since  by \ref{L:naturalf} both are functions, the conclusion now follows.
\end{proof}


\appendix
\renewcommand\thesection{\hspace{-4pt}}

\section{A remark on the syntactic sets of $\teo(\topos(S))$.}\label{A:inter}

\renewcommand\thesection{\Alph{section}}

The purpose of this section is to prove \ref{P:natcanint}, where the relationship between the natural and the canonical interpretations is given explicitly in terms of the logical functor associated to the canonical translation. 

The {\em canonical inclusion} $\iota_\E:\E\rightarrow\topos(\mathrm{Th}(\E))$ is defined for $\E$-objects $A$ as $\iota_\E A:=U_\typea$ and for  $\E$-arrows $A\xrightarrow{f}B$ as $\iota_\E f:=(x\mapsto\typef(x)):U_\typea\bm{\rightarrow}U_\mathbf{B}$.

As part of the proof of Bell's Equivalence Theorem, he constructs its {\em canonical equivalence} $\rho_\E:\topos(\mathrm{Th}(\E))\rightarrow\E$ as follows: For $\bm{\{}x\bm{:}\alpha\bm{\}}$ of type $\mathbf{PA}$ in $\topos(\mathrm{Th}(\E))$, let $\rho_\E \bm{\{}x\bm{:}\alpha\bm{\}}:=\dom\overbar{\llbracket\alpha\rrbracket_x}$ in $\E$. For $\bm{\{}x\bm{:}\alpha\bm{\}}\overset{f}{\bm{\rightarrow}}\bm{\{}y\bm{:}\beta\bm{\}}$
in $\topos(\mathrm{Th}(\E))$, let $\rho_\E f:\rho_\E \bm{\{}x\bm{:}\alpha\bm{\}}\bm{\rightarrow}\rho_\E \bm{\{}y\bm{:}\beta\bm{\}}$ be the unique $g$ such that 
\begin{equation}\label{E:definingGf}
\vdash_{\mathrm{Th}(\E)}\langle\bm{i_{\{x:\alpha\}}}(u),\bm{i_{\{y:\beta\}}}(\bm{g}(u))\rangle\bm{\in}\lvert f\rvert.
\end{equation}
If $X$ is an $S$-set and $\mathfrak{X}$ is an $\teo(\topos(S))$-set of type $\bm{PX}$, let $j_{\rho_{\E}\mathfrak{X}}=(\lvert i_{\rho_{\E}\mathfrak{X}}\rvert,\rho_{\E}\mathfrak{X},X)$. It also follows that the corestriction of $(u\mapsto \bm{j_{\rho_\E(\mathfrak{X})}}(u))$,
\begin{equation}\label{E:epsilon} 
s_{\bm{\rho_\E(\mathfrak{X})}}:U_{\bm{\rho_\E(\mathfrak{X})}}\longrightarrow \mathfrak{X},
\end{equation}
is an isomorphism (cf. \cite{MR972257}).

The following proposition gives an explicit description of how $\teo(\topos(S))$ produces no extraneous objects to those in $\topos(S)$. 
\begin{prop}\label{P:natcanint}
Let $S$ be a local set theory in a local language $\llocal$. Let $X=\{x:\alpha\}$ be an $S$-set of type $\mathbf{PA}$ and $\mathfrak{X}=\{u:\gamma\}$ a $\teo(\topos(S))$-set with $u$ of type $\mathbf{X}$, and set $\rho_{\teo(S)}\mathfrak{X}=\{x:(\llbracket\gamma\rrbracket_{H_{\topos(S)},u})^\natural(x)\}$ in $S$. Then,

\begin{equation}\label{E:canrep}
\eta_S(\rho_{\teo(S)}\mathfrak{X})\cong \mathfrak{X}.
\end{equation}
Furthermore,
\begin{equation}\label{E:naturalcanonical}\llbracket\gamma\rrbracket_{K_{\teo(\topos(S))},u}=\topos(\eta_S)\left(\llbracket\gamma\rrbracket_{H_{\topos(S)},u}\right)\circ (u\mapsto\bm{i_X}(u)).
\end{equation}
\end{prop}
\begin{proof} As per usual, $\topos(\eta_S)$ is represented as an inclusion.

Denote $\llbracket\gamma\rrbracket_{H_{\topos(S)},u}^\natural(x)$ by $\gamma^\natural$. By \ref{L:Xisalphatoo}, the following is a pullback diagram:
\[\xymatrix{
\rho_{\teo(S)}\mathfrak{X}\ar[r]\ar[d]_{i_{\rho_{\teo(S)}\mathfrak{X}}} & 1\ar[d]^\top \\
A\ar[r]_{(x\mapsto\gamma^\natural)} & \subob.
}\]
Since it is clear that $x\in \rho_{\teo(S)}\mathfrak{X}\vdash_Sx\in X$, the following is also readily verified to be a pullback diagram:
\begin{equation}\label{E:natdia}
\xymatrix{
\rho_{\teo(S)}\mathfrak{X}\ar[rr]\ar[d]_{j_{\rho_{\teo(S)}\mathfrak{X}}} & & 1\ar[d]^\top \\
X\ar[rr]_{\llbracket\gamma\rrbracket_{H_{\topos(S)},u}} & & \subob.
}
\end{equation}
On the other hand, also by \ref{L:Xisalphatoo}, the following is a pullback diagram:
\[\xymatrix{
\mathfrak{X}\ar[r]\ar[d]_{i_{\mathfrak{X}}} & 1\ar[d]^\top \\
U_{\mathbf{X}}\ar[r]_{(u\mapsto\gamma)} & \subob.
}\]
Observe that $\rho_{\topos(S)}i_{\mathfrak{X}}=j_{\rho_{\teo(S)}\mathfrak{X}}$, and thus, $\iota_{\topos(S)}\rho_{\topos(S)}i_{\mathfrak{X}}=(u\mapsto\bm{j_{\rho_{\teo(S)}\mathfrak{X}}}(u))$. Therefore, by \eqref{E:epsilon}, \ref{L:Ssetisomorphictoitssymboltype} and \ref{P:functionisotoitsrepresentation}, the following diagram commutes:
\[
\xymatrix{
\mathfrak{X}\ar[d]_{i_{\mathfrak{X}}} & & U_{\bm{\rho_{\teo(S)}\mathfrak{X}}}\ar[ll]^\cong_{s_{\rho_{\teo(S)}\mathfrak{X}}}\ar[d]^{(u\mapsto\bm{j_{\rho_{\teo(S)}\mathfrak{X}}}(u))}\ar[rr]_\cong^{r_{\rho_{\teo(S)}\mathfrak{X}}} & & \rho_{\teo(S)}\mathfrak{X}\ar[d]^{j_{\rho_{\teo(S)}\mathfrak{X}}} \\
U_{\mathbf{X}} & & U_{\mathbf{X}}\ar[ll]^1\ar[rr]^\cong_{r_X} & & X.
}
\]

This finishes the proof of \eqref{E:canrep}. Lastly, let $h$ be the inverse of $r_X=(u\mapsto\bm{i_X}(u)):U_{\mathbf{X}}\rightarrow X$ and consider the following commutative diagram:
\[\xymatrix{
\rho_{\teo(S)}\mathfrak{X}\ar[r]^\cong\ar[d]_{j_{\rho_{\teo(S)}\mathfrak{X}}} & U_{\bm{\rho_{\teo(S)}\mathfrak{X}}}\ar[rr]^\cong\ar[d]^{(u\mapsto\bm{j_{\rho_{\teo(S)}\mathfrak{X}}}(u))} & & \mathfrak{X}\ar[r]\ar[d]^{i_{\mathfrak{X}}} & 1\ar[d]^\top \\
X\ar[r]_h & U_{\mathbf{X}}\ar[rr]_1 & & U_{\mathbf{X}}\ar[r]_{(u\mapsto\gamma)} & \subob.
}\]
Since the right square is a pullback, the top arrows of the middle and left squares are iso and $j_{\rho_{\teo(S)}\mathfrak{X}}$ is monic, the exterior diagram is a pullback. Wherefrom, by \eqref{E:natdia}, \eqref{E:naturalcanonical} follows.
\end{proof}


\bibliographystyle{plainnat}
\bibliography{../ref}

\begin{thebibliography}{5}
\providecommand{\natexlab}[1]{#1}
\providecommand{\url}[1]{\texttt{#1}}
\expandafter\ifx\csname urlstyle\endcsname\relax
  \providecommand{\doi}[1]{doi: #1}\else
  \providecommand{\doi}{doi: \begingroup \urlstyle{rm}\Url}\fi

\bibitem[Bell(1988)]{MR972257}
J.~L. Bell.
\newblock \emph{Toposes and local set theories}, volume~14 of \emph{Oxford
  Logic Guides}.
\newblock The Clarendon Press, Oxford University Press, New York, 1988.
\newblock ISBN 0-19-853274-1.
\newblock An introduction, Oxford Science Publications.

\bibitem[Johnstone(1977)]{MR0470019}
P.~T. Johnstone.
\newblock \emph{Topos theory}.
\newblock Academic Press [Harcourt Brace Jovanovich, Publishers], London-New
  York, 1977.
\newblock ISBN 0-12-387850-0.
\newblock London Mathematical Society Monographs, Vol. 10.

\bibitem[Lambek and Scott(1986)]{MR856915}
J.~Lambek and P.~J. Scott.
\newblock \emph{Introduction to higher order categorical logic}, volume~7 of
  \emph{Cambridge Studies in Advanced Mathematics}.
\newblock Cambridge University Press, Cambridge, 1986.
\newblock ISBN 0-521-24665-2.

\bibitem[Mac~Lane and Moerdijk(1994)]{MR1300636}
Saunders Mac~Lane and Ieke Moerdijk.
\newblock \emph{Sheaves in geometry and logic}.
\newblock Universitext. Springer-Verlag, New York, 1994.
\newblock ISBN 0-387-97710-4.
\newblock A first introduction to topos theory, Corrected reprint of the 1992
  edition.

\bibitem[Rasiowa and Sikorski(1963)]{MR0163850}
Helena Rasiowa and Roman Sikorski.
\newblock \emph{The mathematics of metamathematics}.
\newblock Monografie Matematyczne, Tom 41. Pa\'{n}stwowe Wydawnictwo Naukowe,
  Warsaw, 1963.

\end{thebibliography}

\end{document}